\documentclass[10pt]{amsart}
\pdfoutput=1
\usepackage{amsmath, amsthm, amssymb}
\usepackage{ifpdf}
\usepackage[pdftex]{graphicx}
\usepackage{tikz}
\usetikzlibrary{matrix,arrows,calc}
\usepackage[pdftex,plainpages=false,hypertexnames=false,pdfpagelabels]{hyperref}
 \theoremstyle{plain}
 \newtheorem{thm}{Theorem}[section]
 \newtheorem{cor}[thm]{Corollary}
 \newtheorem{lem}[thm]{Lemma}
 \newtheorem{prop}[thm]{Proposition}
 \theoremstyle{definition}
 \newtheorem{defn}[thm]{Definition}
 \newtheorem{notation}[thm]{Notation}
 \newtheorem{ex}[thm]{Example}
 \theoremstyle{remark}
 \newtheorem{rmk}[thm]{Remark}
 \DeclareMathOperator{\id}{id}
 \DeclareMathOperator{\EFT}{{\rm -}{\sf EFT}}
 \DeclareMathOperator{\EB}{{\rm -}{\sf EB}}
  \DeclareMathOperator{\Hess}{Hess}
\DeclareMathOperator{\SM}{\underline{\sf SM}}
\DeclareMathOperator{\Pf}{{\rm Pf}}
\DeclareMathOperator{\Euc}{\underline{\sf Euc}}
\DeclareMathOperator{\euc}{\underline{\mathfrak euc}}
\DeclareMathOperator{\Diff}{\underline{\rm Diff}}
\DeclareMathOperator{\pt}{\rm pt}
\DeclareMathOperator{\ev}{\rm ev}
\DeclareMathOperator{\cl}{\rm cl}
\DeclareMathOperator{\odd}{\rm odd}
\DeclareMathOperator{\Sym}{\rm Sym}
\DeclareMathOperator{\Vect}{\underline{\rm \Vect}}
\newcommand{\sq}{\mathord{/\!/}}
\def\beq{\begin{eqnarray}}
\def\eeq{\end{eqnarray}}
\def\downin{\ensuremath{\rotatebox[origin=c]{90}{$\in$}}}
\DeclareMathOperator{\dR}{\rm dR} 
 \newcommand{\bp}{\begin{proof}[\ensuremath{\mathbf{Proof}}]}
 \newcommand{\bs}{\begin{proof}[\ensuremath{\mathbf{Solution}}]}
 \newcommand{\ep}{\end{proof}}
 \newcommand{\be}{\begin{enumerate}}
 \newcommand{\ee}{\end{enumerate}}
 \newcommand{\N}{\mathbb{N}}
 \newcommand{\R}{\mathbb{R}}
 \newcommand{\C}{\mathbb{C}}
 \newcommand{\Z}{\mathbb{Z}} 
 \newcommand{\A}{\mathcal{A}}
 \newcommand{\G}{\mathcal{G}}

 \newcommand{\op}{^{\sf{op}}}

\begin{document}

\title{The Chern-Gauss-Bonnet Theorem via Supersymmetric Euclidean Field Theories}

\author{Daniel Berwick-Evans}

\address{Department of Mathematics, Stanford University, Stanford, CA 94305}

\email{danbe@stanford.edu}

\date{\today}

\begin{abstract}
We prove the Chern-Gauss-Bonnet Theorem using sigma models whose source supermanifolds have super dimension~$0|2$. Along the way we develop machinery for understanding manifold invariants encoded by families of $0|\delta$-dimensional Euclidean field theories and their quantization. 
\end{abstract}

\maketitle 
\setcounter{tocdepth}{1}
\tableofcontents

\section{Introduction and Outline of Results}

In this paper we prove the Chern-Gauss-Bonnet theorem using ideas from supersymmetric sigma models. For a closed Riemannian  manifold~$X$, let $\SM(\R^{0|2},X)$ denote the supermanifold of maps from the odd plane, $\R^{0|2}$, into~$X$. We construct a function that associates to a map~$\phi\in \SM(\R^{0|2},X)$ a superspace analog of the energy,~$\mathcal{S}_0(\phi)=\int_{\R^{0|2}}\| T\phi\|^2$, where $T\phi$ denotes the differential of~$\phi$, the norm squared uses the metric on $X$, and the integral is a Berezinian integral.\footnote{Defining $\mathcal{S}_0(\phi)$ precisely requires that we work with $S$-families of maps, $\phi\colon S\times \R^{0|2}\to X$ for $S$ a base supermanifold, as will describe later.} Our first result relates this function to the Euler characteristic of~$X$. 

\begin{thm} The Berezinian integral of $\exp(-\mathcal{S}_0(\phi))$ over $\SM(\R^{0|2},X)$ computes the Euler characteristic of~$X$:
$$
(2\pi)^{-n/2}\int_{\SM(\R^{0|2},X)} \exp(-\mathcal{S}_0(\phi))=\chi(X).
$$
Moreover, the left hand side can be locally identified with an integral of the Pfaffian of the curvature of the Levi-Civita connection on~$X$. 
\label{thm1}
\end{thm}
The above employs a canonical trivialization of the Berezinian line of $\SM(\R^{0|2},X)$ (i.e., there is a canonical volume form on this super space) so the function $\exp(-\mathcal{S}_0(\phi))$ can be integrated over $\SM(\R^{0|2},X)$. The next construction considers a modification of the above, where we take $\mathcal{S}_h(\phi)=\int_{\R^{0|2}} (\|T\phi\|^2-\phi^*h)$ for $h\in C^\infty(X)$ a smooth function. Define the \emph{partition function} as
$$
Z_X(g,h):=(2\pi)^{-n/2} \int_{\SM(\R^{0|2},X)} \exp(-\mathcal{S}_h(\phi)),
$$
where $g$ is the metric on $X$. 

\begin{thm} Let $h\in C^\infty(X)$ be a Morse function and $\lambda\in \R_{>0}$ be a parameter. Then 
$$
 \lim_{\lambda\to\infty} Z_X(g,\lambda h)={\rm Index}(\nabla h)
$$
where the right hand side is the Hopf index of the gradient vector field $\nabla h$. \label{thm2} \end{thm}

The above results identify the two sides of the Chern-Gauss-Bonnet Theorem. In order to equate them we will understand these integrals as coming from a larger structure, namely a quantization procedure for $0|2$-dimensional Euclidean field theories; making this precise encompasses a large part of our work, and sketching the approach is the goal of Sections~\ref{sec:pFTdef} through~\ref{sec:partfunconc} below. The punchline is that the general structure of quantization for $0|\delta$-dimensional Euclidean field theories forces the following.  

\begin{cor} The number $Z_X(g,h)$ is independent of the metric $g$ and the function $h$. \label{thm3} \end{cor}

\begin{rmk} As we will explain in Section \ref{subsec:quant}, the partition function can be viewed as the total volume of the \emph{smooth stack} of fields equipped with a Weinstein volume form determined by the exponentiated classical action. Our results can be rephrased as showing the total volume is independent of the choice of metric $g$ and smooth function $h$, and is equal to the Euler characteristic of $X$. \end{rmk}

From the above three results we deduce the following form of the Chern-Gauss-Bonnet formula. 

\begin{cor}[Chern-Gauss-Bonnet]\label{cor:CGB} Let $R$ denote the Riemann curvature tensor associated to the Levi-Civita connection on a closed Riemannian manifold~$X$, and let ${\rm Pf}(R)$ denote the Pfaffian density of the curvature. Then
$$
(2\pi)^{-n/2}\int_X{\rm Pf}(R)={\rm Index}(\nabla h)
$$
where $h$ is any Morse function on $X$.
\end{cor}

Ours is not the first proof of the Chern-Gauss-Bonnet Theorem using techniques from quantum field theory. The first physical proofs are due to Alvarez-Gaume~\cite{Alvarez} and Witten~\cite{susymorse}. This inspired mathematical arguments by Getzler using heat kernels (e.g., see~\cite{BGV} or \cite{roe}) and also by Lott~\cite{lott_susy}. There were also more algebraic approaches such as the Mathai-Quillen formalism~\cite{mathai-quillen}. Morally, all of these proofs compute an integral over free loop space---as an infinite-dimensional manifold---using various combinations of analysis and physical reasoning. This is motivated by the path integral in $1|2$-dimensional (alias, $N=2$ supersymmetric) quantum mechanics which defines a certain $1|2$-dimensional Euclidean field theory. The approach in this paper is to study a closely related $0|2$-dimensional Euclidean field theory, which keeps all spaces of fields finite-dimensional. Consequently the functional integral that defines quantization is just an ordinary (Berezinian) integral. In this sense, our proof identifies a particular bridge between Chern's original argument~\cite{CGB}---which manifestly takes place in finite-dimensional geometry---and supersymmetric field theory arguments.

This paper also serves as an investigation into the simplest kind of supersymmetric Euclidean field theories and their quantization, following the work of Hohnhold, Kreck, Stolz and Teichner \cite{HKST}. When compared to their higher-dimensional (and higher categorical) cousins these $0|\delta$-dimensional examples appear quite trivial. However, what they lack in richness they make up for in computability, and concrete calculations in supergeometry allow us to examine particular salient features. For example, Proposition~\ref{concord} shows a way in which supersymmetry is essential if one wishes to obtain interesting topological invariants from $0|\delta$-dimensional Euclidean field theory over manifolds, and quantization of families of $0|2$-dimensional field theories (see Theorem~\ref{02push}) requires that we restrict attention to \emph{renormalizable} families (see Definitions~\ref{rmk:RG}-\ref{def:renorm}). It remains mysterious how or if these features generalize in higher dimensions. 

\subsection{The Chern-Gauss-Bonnet Theorem as localization}\label{CGB}

In this subsection we outline the proofs of Theorems \ref{thm1} and \ref{thm2}. We begin with a geometric characterization of the relevant mapping space.
\begin{lem} \label{geo}
Given a connection on $X$, there exists an isomorphism of supermanifolds
$$\underline{\sf SM}(\R^{0|2},X)\cong p^*(\pi (T X\oplus TX))$$
where $p\colon  TX\to X$ is the usual projection. Hence, after a choice of connection, a point in $\underline{\sf SM}(\R^{0|2},X)$ is a point of $X$, two odd tangent vectors, and one even tangent vector; we denote this quadruple as $(x,\phi_1,\phi_2,F)$. 
\end{lem}

Let $h\in C^\infty X$. The classical action function on $\SM(\R^{0|2},X)$ is defined as 
$$\mathcal{S}_h(\Phi):=\int_{\R^{0|2}} \left(\frac{1}{2}\|T\Phi\|^2-\Phi^*h\right)$$ 
for a definition of $\|T\Phi\|^2$ to be given in Section \ref{sec:Lagdens}. The previous lemma allows us to express  this function in terms of familiar geometric data on~$X$. The following is a technical computation done in Section \ref{sec:push}.

\begin{lem}\label{action} The action functional for the $0|2$-sigma model with potential~$h$ evaluated at a point $\Phi=(x,\phi_1,\phi_2,F)\in \underline{\sf SM}(\R^{0|2},X)$ is
\beq
\mathcal{S}_h(\Phi)&= &\frac{1}{2}\|F\|^2/2+\frac{1}{2}R(\phi_1,\phi_2,\phi_1,\phi_2)/2-\langle F,\nabla h\rangle-\Hess(h)(\phi_1,\phi_2),\nonumber
\eeq
where $\Hess(h)$ denotes the covariant Hessian of the function $h$. 
\end{lem}
With this in hand, we insert a parameter $\lambda\in \R$ in front of $h$, denoting the resulting 1-parameter family of action functions by $\mathcal{S}_{\lambda h}$. We calculate the partition function,
$$
Z_X(g,\lambda h) =(2\pi)^{-n}\int_{\underline{\sf SM}(\R^{0|2},X)} \exp(-\mathcal{S}_{\lambda h}(\Phi))\mathcal{D}\Phi,
$$
by first integrating over the fibers in the horizontal direction in the diagram
$$
\begin{tikzpicture}[>=latex]
\node (A) at (0,0) {$\SM(\R^{0|2},X)$};
\node (B) at (3,0) {$ \pi (TX\oplus TX)$};
\node (C) at (0,-1.5) {$TX $};
\node (D) at (3,-1.5) {$X$};
\draw[->] (A) -- (B);
\draw[->] (A) -- (C);
\draw[->] (B) -- (D);
\draw[->] (C) to node [above=1pt] {$p$} (D);
\draw (.5,-1) -- (.5,-.5) -- (1,-.5);
\fill +(.75,-.75) circle (.04);

\end{tikzpicture}
$$
which amounts to a Gaussian integral in the $F$-variable. The result is
$$
Z_X(g, \lambda h)=(2\pi)^{-n/2}\int_{\pi (TX\oplus TX)} \exp\left(-\frac{\lambda^2}{2}||\nabla h||^2+\lambda\Hess(h)(\phi_1,\phi_2)-R(\phi_1,\phi_2,\phi_1,\phi_2)\right).
$$
From here our argument is very similar in structure to Mathai and Quillen's proof of the Chern-Gauss-Bonnet Theorem~\cite{mathai-quillen}, though our particular super-geometric setting introduces a few wrinkles. If we set $\lambda=0$, the Berezinian integral defining $Z_X(g, \lambda h)$ will first project~$\exp(-R)$ onto the top component\footnote{Strictly speaking, $\exp(-R)$ has no top component; however, the Berezinian integral picks out the relevant component of $\exp(-R)$ to obtain the Pfaffian of $R$.} and a computation in Section \ref{sec:int} identifies this with the Pfaffian of the curvature so that
$$
Z_X(g,0)=(2\pi)^{-n/2}\int_X \Pf(R),
$$
recovering one side of the Chern-Gauss-Bonnet formula and leading to Theorem \ref{thm1}. As~$\lambda\to \infty$, we will show that the integral for $Z_X(g,\lambda h)$ is supported on a neighborhood of the set where~$\nabla h=0$. Assuming $h$ is Morse, the value of the limit can be computed straightforwardly,
$$
\lim_{\lambda\to\infty} Z_X(g,\lambda\cdot h)= \sum_{\nabla h=0} {\rm index}(\nabla h)=\chi(X),
$$
i.e., the integral computes the Hopf index of $\nabla h$, leading to Theorem \ref{thm2}. 

Applying Theorem \ref{thm3}, we find
$$
(2\pi)^{-n/2}\int_X \Pf(R)=Z_X(g,0)=\lim_{\lambda\to \infty}Z_X(g, \lambda h)=\sum_{\{\nabla h=0\}} {\rm index}(\nabla h),
$$
which is the Chern-Gauss-Bonnet formula. 

There is a physical interpretation of the above argument as a toy-model of a path integral localization: the integral of the Pfaffian comes from an integral over all fields, whereas the sum over critical points of $h$ is an integral over a formal neighborhood of the classical solutions of the sigma model action. The former can be thought of as a 0-dimensional ``path" integral, whereas the latter is a stationary phase approximation. This kind of localizing behavior in quantization procedures exists for other Euclidean field theories with two supersymmetries. For example the $1|2$-Euclidean field theory described by Witten in~\cite{susymorse} shows how the de~Rham complex localizes onto the Morse complex. 

\subsection{$0|\delta$-Euclidean field theories}\label{sec:pFTdef}

We give a brief and incomplete introduction to supersymmetric Euclidean field theories (EFTs) in the style of Stolz and Teichner. The full-blown definition is both lengthy and unfinished (see \cite{ST11}) but follows in the footsteps of Michael Atiyah, Maxim Kontsevich and Graeme Segal's approaches to field theories. Namely, a field theory is a symmetric monoidal functor from the $d$-category of $d|\delta$-dimensional  Euclidean bordisms over $X$ to some algebraic $d$-category, $d\hbox{\rm -}{\sf ALG}$:
$$
d|\delta\EFT(X):={\sf Fun}^\otimes_{\sf SM}(d|\delta\EB(X),d\hbox{\rm -}{\sf ALG}),
$$
and we require these functors to be fibered over supermanifolds so that the resulting field theories are, in an appropriate sense, \emph{smooth}; the notation we use for functors fibered over the category ${\sf SM}$ of supermanifolds is ${\sf Fun}_{\sf SM}$. The algebraic target is the familiar category of real vector spaces when $d=1$, some delooping (or categorification) thereof for $d>1$, and the looping (or decategorification) of vector spaces for $d=0$, namely the commutative monoid $(\underline{\R},\times)$, thought of as a symmetric monoidal 0-category. Even the naive definition of $d|\delta\EB(X)$ is quite intricate: it should be a $d$-category internal to symmetric monoidal stacks whose stack of $k$-morphisms are comprised of bundles with fiber $k|\delta$-Euclidean supermanifolds equipped with a map to~$X$. The symmetric monoidal structure comes from disjoint union in the fiber direction. Compositions should be given by gluing these $k|\delta$-manifolds in a way that respects the Euclidean geometries, which suggests that all the $k|\delta$-Euclidean manifolds be collared. These details have yet to be understood completely, except in some low-dimensional examples \cite{HKST,mingeo}. There is a related version of the above in which the bordism category is comprised of cs-manifolds and the target category deloops vector spaces over~$\C$; e.g., see \cite{ST11}. The approach of this paper is to stick to the \emph{very} low-dimensional theories: in dimensions $0|\delta$ the higher categorical complexities disappear. In fact, as we will sketch below, $0|\delta$-field theories are just certain functions on a finite dimensional supermanifold.

To be more precise, in \cite{HKST} the authors show that the internal 0-category of $0|\delta$-Euclidean bordisms over $X$ can be thought of as the free symmetric monoidal category on the connected bordisms. This comes from the geometric fact that every $0|\delta$-dimensional supermanifold is a coproduct of connected ones. We denote this category by $0|\delta\EB_{\rm conn}(X)$ (with ``conn" standing for ``connected"), and emphasize that it has no symmetric monoidal structure. Furthermore, this 0-category internal to stacks has a presentation by the quotient groupoid in supermanifolds
$$
0|\delta\EB_{\rm conn}(X)\cong \underline{\sf SM}(\R^{0|\delta},X)\sq\Euc(\R^{0|\delta}), 
$$
where $\underline{\sf SM}(\R^{0|\delta},X)$ denotes the inner hom in generalized supermanifolds (see Section \ref{sec:not}), and $\Euc(\R^{0|\delta})<\underline{\sf Diff}(\R^{0|\delta})$ is a chosen group which we call the \emph{Euclidean isometries} of~$\R^{0|\delta}$. To be explicit, the generalized supermanifold of objects is $\SM(\R^{0|\delta},X)$ and the morphisms are $\SM(\R^{0|\delta},X)\times \Euc(\R^{0|\delta})$. The source and target maps are given by the projection and precomposition, respectively, and the unit includes along the identity element of~$\Euc(\R^{0|\delta})$. To make this geometric picture precise requires one to use the functor of points, which we review briefly at the end of this section. We observe that~$\Euc(\R^{0|\delta})$ is the internal automorphism group of the unique object of~$0|\delta\EB_{\rm conn}(\pt)$. 

With a little work, one can understand fibered functors to $\underline{\R}$ as functions, 
$$
0|\delta\EFT(X)={\sf Fun}^\otimes_{\sf SM}(0|\delta\EB(X),\underline{\R})\cong {\sf Fun}_{\sf SM}(0|\delta\EB_{\rm conn}(X),\underline{\R})\cong  C^\infty(0|\delta\EB_{\rm conn}(X))^{\ev},
$$
where $\underline{\R}$ is the representable stack given by the supermanifold $\R$ and has the structure of a commutative monoid in stacks $(\underline{\R},\times)$ when considering symmetric monoidal functors. The first isomorphism uses the fact that the free functor is left adjoint to the forgetful functor. Functions on a groupoid are the invariant functions, $C^\infty(0|\delta\EB_{\rm conn}(X) )\cong C^\infty(\SM(\R^{0|\delta},X))^{\Euc(\R^{0|\delta})}$, and for this paper we take this as a definition of field theories:
$$
0|\delta\EFT(X):=\left(C^\infty(\SM(\R^{0|\delta},X))^{\Euc(\R^{0|\delta})}\right)^{\ev}.
$$ 

It will be useful that the inner hom above is represented by a supermanifold; if we choose an isomorphism $\R^{0|\delta}\cong (\R^{0|1})^\delta$, we can iterate the isomorphism $\SM(\R^{0|1},X)\cong \pi TX,$ (see Example \ref{ex:piTX}) to obtain $
\underline{\sf SM}(\R^{0|\delta},X)\cong (\pi T)^\delta X,$ where $\pi T\colon  {\sf SM}\to {\sf SM}$ is the functor that takes a supermanifold to the total space of its odd tangent bundle, and $(\pi T)^\delta$ denotes~$\delta$ applications of this functor. Thus, $0|\delta\EB_{\rm conn}(X)$ admits a description by a quotient groupoid in \emph{supermanifolds}, not just generalized ones. In particular, $0|\delta$-dimensional field theories are functions on the supermanifold $\SM(\R^{0|\delta},X)$ invariant under the action of the Euclidean group. We will now unpack this for some examples. 

\begin{ex}[$0|0\EFT(X)$] We have that $\R^{0|0}\cong \pt$, and so 
$$
\Euc(\R^{0|0}):=\{\id \} \cong \underline{\sf Diff}(\R^{0|0}).
$$ 
Hence, 
$$
0|0\EFT(X)\cong C^\infty(\underline{\sf SM}(\R^{0|0},X)\sq\{ \id\})^{\ev}\cong C^\infty(X).
$$
\end{ex}
\begin{ex}[{$0|1\EFT(X)$}] Following \cite{HKST}, we choose the Euclidean group
$$
\Euc(\R^{0|1}):=\R^{0|1}\rtimes \Z/2<\underline{\sf Diff}(\R^{0|1})\cong \R^{0|1}\rtimes \R^\times,
$$
and use the fact that $C^\infty(\SM(\R^{0|1},X))\cong C^\infty(\pi TX)\cong \Omega^\bullet(X),$ where differential forms are regarded as being $\Z/2$-graded via mod 2 reduction of the usual de~Rham grading. We claim
$$
0|1\EFT(X)\cong C^\infty(\pi TX\sq(\R^{0|1}\rtimes \Z/2))^{\ev}\cong \Omega^{\ev}_{\cl}(X),
$$
where $\Omega_{\cl}^{\ev}$ denotes the sheaf of closed, even differential forms. Indeed, functions on the quotient groupoid are functions on $\pi TX$ invariant under the group action; an exercise (sketched in Example \ref{0|1deRham}, also see \cite{HKST}) shows that the infinitesimal action of $\R^{0|1}$ on~$\Omega^\bullet(X)$ is precisely the de~Rham $d$, and the $\Z/2$ action is by the grading involution. Hence, functions fixed under these two actions are $d$-closed and of even degree. 
\end{ex}

To obtain the odd forms we require \emph{twisted} field theories. Again, there is a general definition (see \cite{HKST,ST11}) but we will be content to specialize to dimensions~$0|\delta$. A {\it twist} is a line bundle on $0|\delta\EB_{\rm conn}(X)$, and a {\it twisted field theory} is a section of this line bundle. Note that a section of the trivial line bundle is just a function, which was our notion of an (untwisted) field theory above. In summary, we define
$$
0|\delta\EFT^\mathcal{L}(X):= \Gamma(\underline{\sf SM}(\R^{0|\delta},X)\sq\Euc(\R^{0|\delta}),\mathcal{L})^{\rm ev}.
$$
Typically we want twists to be naturally defined for any manifold $X$, in a sense explained in the next paragraph. 

For any smooth map $f\colon  X\to Y$ we have a functor $0|\delta\EB_{\rm conn}(X)\to 0|\delta\EB_{\rm conn}(Y)$ which in turn induces $f^*\colon  0|\delta\EFT^\mathcal{L}(Y)\to 0|\delta\EFT^{f^*\mathcal{L}}(X)$. Hence, if $\mathcal{L}$ is a line bundle on $0|\delta\EB_{\rm conn}(\pt)$, the canonical map $p\colon X\to \pt$ produces a line bundle $p^*\mathcal{L}$ on $0|\delta\EB_{\rm conn}(X)$ for each $X$. Since $0|\delta\EFT(\pt)\cong \pt\sq \Euc(\R^{0|\delta})$, the groupoid of such line bundles is equivalent to one whose objects are homomorphisms of super Lie groups, $\rho\colon  \Euc(\R^{0|\delta})\to \R^\times$. When such a 1-dimensional representation $\rho$ is fixed, $\mathcal{L}_\rho$ will denote the corresponding line bundle. Sections of $p^*\mathcal{L}_\rho$ are functions on $\underline{\sf SM}(\R^{0|\delta},X)$ that are \emph{equivariant} with respect to the induced action of $\Euc(\R^{0|\delta})$ on $\R$ via $\rho$. Indeed, in the more general situation of line bundles over a quotient groupoid $M\sq G$ coming from $\rho\colon  G\to \R^\times$ a homomorphism, we have (see Corollary 39 of \cite{HKST})
$$
\Gamma(M\sq G,\mathcal{L}_\rho)\cong \{x\in C^\infty(M) \mid \mu^*(x)=p_1^*(x)\cdot p_2^*(\rho)\in C^\infty(M\times G)\},
$$ 
$$
\Gamma(M\sq G,\pi\mathcal{L}_\rho)\cong \{x\in C^\infty(M)^{\odd}\mid \mu^*(x)=p_1^*(x)\cdot p_2^*(\rho)\in C^\infty(M\times G)\}
$$ 
where $p_1\colon  M\times G\to M$, $p_2\colon  M\times G\to G$ are the projections and $\mu\colon  M\times G\to M$ is the action. Returning to field theories, the above discussion shows that the assignment $X\mapsto 0|\delta\EFT^{p^*\mathcal{L}_\rho}(X)$ is natural in $X$ and defines a sheaf of vector spaces on the site of manifolds. In fact, we can obtain a sheaf of graded algebras on manifolds whose degree $k$ part on $X$ is $0|\delta\EFT^{p^*\mathcal{L}_\rho^{\otimes k}}(X)$; multiplication in this algebra comes from the tensor product of line bundles. We observe that the degree 0 part is the same as twisted field theories gotten from choosing $\rho$ to be the trivial homomorphism; this gives ordinary (or untwisted) field theories since equivariant functions with respect to the trivial action on $\R$ are exactly the invariant functions. When the line bundle $\mathcal{L}_\rho$ is understood, we use the notation
$$
0|\delta\EFT^\bullet (X):=0|\delta\EFT^{p^*\mathcal{L}_\rho^{\otimes \bullet}}(X).
$$

\begin{ex} When $\delta=1$ following \cite{HKST} we choose the projection 
$$
\rho\colon  \R^{0|1}\rtimes \Z/2\to \Z/2\subset \R^\times,
$$ 
to build a line bundle $\pi \mathcal{L}_\rho$. The functions on $\SM(\R^{0|1},X)$ equivariant with respect to the action of $\R^{0|1}\rtimes \Z/2$ are precisely the closed forms in the $-1$ eigenspace of the grading involution, i.e., the odd forms. Hence
$$
0|1\EFT^{p^*\pi\mathcal{L}_\rho}(X)\cong \Omega^{\odd}_{\cl}(X).
$$
We've sketched the proof of the following result. 

\begin{thm}[Hohnhold-Kreck-Stolz-Teichner] 
There are isomorphisms of abelian groups 
$$
0|1\EFT^k (X)\cong \left\{ \begin{array}{lll} \Omega^{\ev}_{\cl}(X) & \null & k={\rm even,} \\ \Omega^{\odd}_{\cl}(X) & \null & k={\rm odd}.\end{array}\right.
$$
These isomorphisms are compatible with the graded ring structure on both sides, namely tensor products of field theories on the left (i.e., multiplication of functions on $\SM(\R^{0|1},X))$ and wedge products of forms on the right. 
\end{thm}
\end{ex}

\begin{ex} We can generalize the above example to $0|\delta$-Euclidean field theories by declaring $\Euc(\R^{0|\delta}):=\R^{0|\delta}\rtimes O(\delta)$ to be the isometry group, and take $\rho\colon  \R^{0|\delta}\rtimes O(\delta)\to \Z/2\subset \R^\times$ to define a twist where $\rho$ first projects to $O(\delta)$, then applies the determinant homomorphism. When $\delta$ is even, we take $\mathcal{L}_\rho$ as the twist, and when $\delta$ is odd we take $\pi\mathcal{L}_\rho$; this is the choice that leads to nontrivial sections, since $-{\rm id}\in O(\delta)$ acts by the grading involution on functions on $\SM(\R^{0|\delta},X)$. Loosely, these $0|\delta$-dimensional field theories are a generalization of closed differential forms. For a different generalization (that considers the full diffeomorphism group of $\R^{0|\delta}$) see Kochan and \v Severa in \cite{gorms}. \end{ex}

\subsection{Concordance}

Given a closed differential form, one can extract a topological invariant by considering the de~Rham cohomology class that form represents. Analogously, $0|\delta$-EFTs give supergeometric objects generalizing closed forms, and one can extract topological information from these. One way to implement this passage from geometric data to topological data is to take \emph{concordance classes}.

\begin{defn} 
Two twisted field theories $E_+, E_-\in 0|\delta\EFT^k(X)$ are \emph{concordant} if there exists a twisted field theory $\tilde{E}\in 0|\delta\EFT^k(X\times \R)$ such that $i_\pm^*\tilde{E}= \pi_\pm^*E_\pm,$
where 
$$i_\pm\colon  X\times (\pm 1,\pm\infty) \hookrightarrow X\times \R, \quad \pi_\pm\colon  X\times (\pm 1,\pm\infty)\to X$$
are the usual inclusion and projection maps, respectively. We denote the set of field theories up to concordance by $0|\delta\EFT^k[X]$ and the concordance class of a field theory $E$ by $[E]$.\label{def:conc}
\end{defn}

It is easy to check that concordance defines an equivalence relation. In fact, for field theories twisted by a line bundle $\mathcal{L}_\rho$ the ring structure on $0|\delta\EFT^\bullet(X)$ descends to concordance classes. Field theories up to concordance furnish an additive contravariant functor from manifolds to graded rings,
$$
0|\delta\EFT^\bullet[-]\colon  {\sf Man}^{\op}\to {\sf GRing}.
$$ 
It is straightforward to check that smoothly homotopic maps between manifolds are assigned the same homomorphism between graded rings. In particular, the graded ring $0|\delta\EFT^\bullet[X]$ is a homotopy invariant of $X$.

An application of Stokes' Theorem shows the following; see also Proposition \ref{concord}.
\begin{thm}[Hohnhold-Kreck-Stolz-Teichner]
There is an isomorphism of abelian groups
\beq
0|1\EFT^k[X]\cong \left\{ \begin{array}{lll} H_{\dR}^{\ev}(X) & \null & k={\rm even,} \\ H_{\dR}^{\odd}(X) & \null & k={\rm odd,}\end{array}\right.\nonumber
\eeq
compatible with the ring structures on both sides, where $H^\bullet_{\dR}$ is de~Rham cohomology with its usual cup product.\end{thm} 
\begin{rmk} Although $0|2$-dimensional theories share many similar structures with de~Rham cohomology, it turns out that the functor $X\mapsto 0|2\EFT^\bullet[X]$ fails to have Mayer-Vietoris sequences. \end{rmk}

\subsection{Sigma models and a cartoon of quantization}

Classical supersymmetric sigma models are a natural generalization of classical mechanics---rather than studying paths in a Riemannian manifold~$X$, we study maps of $d|\delta$-dimensional manifolds into~$X$. The sigma model action functional generalizes the energy of a path: we use certain geometric structures on the source manifold together with the metric on $X$ to make sense out of~$\int_\Sigma |d\Phi|^2$ for~$\Phi\colon \Sigma\to X$.

To warm-up, we'll review quantization in ordinary mechanics. Let 
$$
\mathcal{F}_tX:=\underline{\sf SM}(S^1_t,X)
$$
denote maps from loops of circumference~$t$, denoted~$S_t^1$, to $X$. Let $g$ be a Riemannian metric on $X$ and $h\in C^\infty(X)$ a smooth function. Define the \emph{classical action} by
$$
\mathcal{S}_h(\gamma):=\int_{S^1_t} \left(\frac{1}{2}\|\dot{\gamma}\|^2-\gamma^*h\right)dt.
$$
A quantization procedure for this classical system is furnished by the path integral, or more precisely, the Wiener measure on the paths in $X$. As shown by Anderson and Driver~\cite{driver} and B\"ar and Pf\"affle~\cite{baer}, one can view this measure as a limit of finite dimensional measures which is how one makes sense out of the formulas like
$$
\langle \mathcal{O} \rangle_t = \int_{\mathcal{F}_tX} \mathcal{O} \frac{\exp(-\mathcal{S}_h)}{N_t}\mathcal{D}_t\gamma,
$$
where $\mathcal{O}\in C^\infty(LX)$ is a function (alias, classical observable), and the Wiener measure is
$$
\mathcal{D}_tW_g=\frac{\exp(-\mathcal{S}_h)}{N_t}\mathcal{D}_t\gamma.
$$
The right hand side cannot be taken literally: the formula reflects the ingredients that go into the finite dimensional approximations to the measure, but a measure $\mathcal{D}_t\gamma$ and normalization $N_t$ do not exist independently.

Although analogous measures for $d>1$ have yet to be constructed rigorously, in many examples there are quantization procedures that mimic the behavior of such a measure. Below we formalize some desired properties and give mathematical examples. There are three important features of the Wiener measure we wish to emphasize and incorporate in the structure. First, the classical energy is an essential ingredient: $D_t\gamma$ and $N_t$ do not exist. Hence, if we hope to construct a quantization procedure we will require some geometric data on $X$ that allows us to write a kinetic energy term in the classical action, which we think of (at least philosophically) as defining a Gaussian measure on the corresponding mapping space. In some sense, the classical sigma model \emph{is} this measure. Second, continuously varying the Riemannian metric on $X$ results in a continuous family of Wiener measures, which allows us to consider quantization in families. For example, a 1-parameter family of metrics on~$X$ results in a 1-parameter family of quantum mechanical theories, and more generally a $Y$-family of metrics results in a $Y$-family of quantum theories. Lastly, the Wiener measure plays nicely with isometries of circles, in the sense that the quantum expectation value $\langle \mathcal{O}\rangle_t$ depends only on the \emph{isomorphism class} of the loop $\gamma\colon S^1_t\to X$, i.e., the parameter~$t$. By varying~$t$ we can assemble this quantum observable into a smooth function on the moduli stack of Euclidean circles. Our definition of quantization attempts to be as general as possible while satisfying these three restrictions. 

We will conclude this subsection by defining sigma models, and in the next two subsections will set up the relevant background to define quantization. 

The primary input for a sigma model is a \emph{stack of fields}
\beq
\mathcal{F}_\sigma X:=\underline{\sf SM}(\Sigma^{d|\delta},X)\sq \Euc(\Sigma_\sigma),\label{eq:fields}
\eeq
for $\sigma$ a chosen geometry on $\Sigma$, and for simplicity we are assuming that $\Sigma$ is closed. Following the example from mechanics, we want a classical action on $\mathcal{F}_\sigma X$ invariant under this isometry group, which is precisely a function on the stack. This prompts a definition.

\begin{defn} A $d|\delta$-dimensional (classical) sigma model is a function $\mathcal{S}\in C^\infty(\SM(\Sigma^{d|\delta},X)$ depending on a metric $g$ and smooth function $h$ on $X$ that is invariant under the action of the isometry group of $\Sigma$ with geometry $\sigma$. 
\end{defn}

Typically the function defining the classical sigma model will be of the form
$$
\mathcal{S}_h(\Phi):=\int_{\Sigma}  \left(\frac{1}{2}\|T\Phi\|^2-\Phi^*h\right){\rm vol}_{\Sigma,\sigma}, \quad \Phi\in \mathcal{F}_\sigma X,
$$
where the first term is a \emph{kinetic} term that computes something like the energy of a map and the second term is a \emph{potential} term that pulls back a function on~$X$ and integrates it on~$\Sigma$. 

In the dimensions where they exist, classical sigma models are fairly well-understood mathematically (see Freed \cite{5lectures}), but as we move into higher dimensions and add more geometric data to $\Sigma$, it is unclear if they can be used to define a Gaussian measure on fields analogous to the Wiener measure. In several examples of \emph{supersymmetric} sigma models the \emph{partition function}
$$
Z_X^{d|\delta}(g,h):=\int_{\mathcal{F}_\sigma X} \frac{\exp(-\mathcal{S}_h(\Phi))}{N_\sigma}\mathcal{D}_\sigma\Phi,
$$ 
frequently constructed through physical reasoning, turns out to be a topological invariant of $X$. Some of the known examples are listed in Table \ref{table}. This motivates our desire to understand the extent to which a rigorous definition of quantization might be lurking in the examples where partition functions encode topology. 

\begin{table}
\begin{tabular}{|c|c|c|c|c|}
\hline
dim &  partition function \\
\hline 
$0|1$ &  0, for $\dim(X)>0$; signed cardinality when dim$(X)=0$ \\
$1|1$ &  $\hat{A}$-genus (Alvarez-Gaume, \cite{Alvarez})\\
$2|1$ &  Witten genus? (Witten, \cite{witten_dirac}; Stolz-Teichner, \cite{ST11}) \\
\hline
$0|2$ &  Euler characteristic (Theorem \ref{thm1}) \\
$1|2$ &  Euler characteristic and signature (Witten, \cite{susymorse}) \\
$2|2$ &  $S^1$-Equivariant signature of $LX$? (Witten, \cite{witten_dirac}) \\
\hline
\end{tabular}
\vspace{.1in}
\caption{In the listed dimensions, supersymmetric sigma models give (or are conjectured to give) the listed invariant. Conjectural statements are marked by ``?". \label{table}}
\end{table}

\subsection{Stacky integrals and renormalizable $0|\delta$-EFTs} \label{subsec:quant}

For sigma models where the source manifold $\Sigma$ is $0|\delta$-dimensional, the space of fields, $\underline{\sf SM}(\R^{0|\delta},X),$ is a \emph{finite}-dimensional supermanifold. So, with a bit of work, we can make the integration defining quantization completely rigorous. We observe that for $\delta>0$, fields can form an interesting stack owing to the nontrivial automorphisms of~$\R^{0|\delta}$. Weinstein has explained integrals on smooth stacks~\cite{weinstein}, and it turns out that our construction can be interpreted in terms of his definition. The stacky aspects end up being comparably easy, so we'll be brief in our description. 

When computing an integral on a stack, one wants to include isotropy in the computation; for example, for a stack $\pt\sq G$ arising from the action of a finite group $G$ on the point, there is a counting measure such that the volume of stack is $1/|G|$. In the case where a (possibly nondiscrete) Lie group $G$ acts on a manifold $M$, Weinstein shows (\cite{weinstein}, Theorem 3.2) that a measure on the stack $M\sq G$ arises from a section of the Berezinian line ${\rm Ber}(\mathfrak{g}) \otimes {\rm Ber}(M)$~on~$M$ invariant under the action of $G$, where $\mathfrak{g}$ is the Lie algebra of $G$ and we use the inclusion of Lie algebras $\mathfrak{g}\to \Gamma(TM)$ given by the action.

In the case of $0|\delta$-dimensional field theories, we have $M=\SM(\R^{0|\delta},X)$ and $G=\Euc(\R^{0|\delta})$. Proposition \ref{intprop} will show that the the Berezinian line on $\SM(\R^{0|\delta},X)$ is canonically trivialized; we denote the trivializing section by $\mathcal{D}\Phi$. The bundle on $\SM(\R^{0|\delta},X)$ coming from the Berezinian of the Lie algebra of $\Euc(\R^{0|\delta})$ is trivial, so any two choices of trivialization differ by a nonvanishing function. There is a standard trivialization coming from our chosen groupoid presentation of the underlying stack; we will use a different trivialization gotten from the standard one using the exponentiated classical action viewed as a nonvanishing function. Since $\mathcal{S}$ is invariant under the action of $\Euc(\R^{0|\delta})$, so is $\exp(-\mathcal{S})D\Phi$ and this data gives us a Weinstein volume form on the stack $\SM(\R^{0|\delta},X)\sq \Euc(\R^{0|\delta})$. 

In our case, Weinstein's formula tells us that to integrate a function we compute
\beq
\int_{\SM(\R^{0|\delta},X)}\omega(\Phi) \exp(-\mathcal{S}(\Phi))\frac{\mathcal{D}\Phi}{N},\label{schematic}
\eeq 
where $N$ is some finite normalization constant into which we've absorbed the volume of~$\Euc(\R^{0|\delta})$. When the above is defined, we are guaranteed that the result is an invariant of the stack $\SM(\R^{0|\delta},X)\sq \Euc(\R^{0|\delta})$ together with our choice of measure. 

However, $\SM(\R^{0|\delta},X)\sq \Euc(\R^{0|\delta})$ is not a proper stack, so a priori there might be very few integrable functions: $\SM(\R^{0|\delta},X)$ is noncompact for $\delta>1$. Indeed, this forces $\mathcal{S}\ne 0$ for $\delta>1$ if we want, e.g., 1 to be integrable---we \emph{require} an interesting action functional for our stack to have finite volume, or (equivalently) to define a partition function. As we will describe, all functions that are polynomial ``at infinity" will be integrable with respect to our chosen measure. This polynomial behavior can be described in terms of the renormalization group action on field theories.

\begin{defn} \label{rmk:RG}The \emph{renormalization group} (RG) \emph{action} on $0|\delta\EFT^\bullet(X)$ is the action induced from dilating $\R^{0|\delta}$ by $\R^\times$, which in turn gives an action on $0|\delta\EB_{\rm conn}(X)$, and hence $0|\delta$-EFTs. \end{defn}

Since $0|\delta$-dimensional field theories over $X$ form a vector space, we can consider the field theories that are in the $\lambda$-eigenspace of the infinitesimal RG-action, i.e., where $r\in\R^\times$ acts by $r^\lambda$. 

\begin{rmk} The action of $\R^\times$ extends to one by the monoid $\R$, so it turns out that $\lambda$ is necessarily a natural number. For our purposes it will suffice to observe that the eigenvalues of the infinitesimal $\R^\times$-action are positive integers. 
\end{rmk}

\begin{defn} Define \emph{polynomial functions} on $\underline{\sf SM}(\R^{0|\delta},X)$ as 
$$
C^\infty_{\rm pol}(\underline{\sf SM}(\R^{0|\delta},X)):=\bigoplus_{k\in \N}  \{f \in C^\infty(\underline{\sf SM}(\R^{0|\delta},X)) \ |\ r\cdot f= r^k f, \ r\in \R_{>0}\},
$$
where $r\cdot f$ denotes the action of $\R_{>0}$ on functions on $\underline{\sf SM}(\R^{0|\delta},X)$ induced by the dilation action of $\R_{>0}$ on $\R^{0|\delta}$. When $r\cdot f=r^kf$, we say $f$ has \emph{polynomial degree $k$}. \label{def:pol}
\end{defn}

\begin{defn} \emph{Renormalizable} twisted $0|\delta$-Euclidean field theories over $X$ are 
$$
0|\delta\EFT_{\rm pol}^\bullet (X):=\bigoplus_{k\in \N}  \{E\in 0|\delta\EFT^\bullet(X) \ |\ r\cdot E= r^k E, \ r\in \R^\times\},
$$
where $r\cdot E$ denotes the action of the renormalization group on field theories. 
\label{def:renorm} \end{defn}

\begin{rmk} One can rephrase the polynomial growth condition by putting $\omega$ in Equation~\ref{schematic} into the exponent: then renormalizablity translates into at most logarithmic growth with the RG action which (ignoring formal details) agrees with renormalizability in Kevin Costello's sense; see Definition 7.2.1 in \cite{costbook}. \end{rmk}

\subsection{Quantization in families}
A key aspect of our definition and construction of quantization is the ability to quantize in families. To set this up, we need to make sense of families of geometric structures. 

Consider the category whose objects are submersions of manifolds $p\colon X\to Y$ and whose morphisms are fiberwise isomorphisms, i.e., smooth maps such that the canonical map to the pullback is an isomorphism. We give this category the structure of a site by declaring a family $\{ f_i\colon p_i\to p\}$ to be a covering family if the map from the coproduct $\coprod_i p_i$ is surjective. Denote this site by ${\sf Subm}$. There is a similar site defined by restricting attention to \emph{proper} submersions; we denote this site by ${\sf pSubm}$.

\begin{defn} 
Let a sheaf of sets $\mathcal{G}$ on the site ${\sf Subm}$ be given. A \emph{$\mathcal{G}$-oriented submersion} is a submersion $p\colon X\to Y$ together with a section $g\in \mathcal{G}(p)$. Similarly, a \emph{$\mathcal{G}$-oriented proper submersion} is a proper submersion with a chosen section. For a manifold $X$, a \emph{$\mathcal{G}$-structure on $X$} is a section of $\mathcal{G}$ applied to the submersion $X\to \pt$.
\end{defn}

\begin{ex} The sheaf ${\sf Vol}$ on the site ${\sf Subm}$ assigns to $p\colon X\to Y$ a smoothly varying family of volume forms, i.e., a nonvanishing section of $\Lambda^{\rm top}(Vp)$ where $Vp$ denotes the vertical tangent bundle to the projection~$p$. \label{ex:geovol}\end{ex}

\begin{ex} The sheaf ${\sf Or}$ on the site ${\sf Subm}$ assigns to a submersion $\pi\colon X\to Y$ the set of smoothly varying orientations on $Vp$. \end{ex}

\begin{ex} The sheaf ${\sf Riem}$ on the site ${\sf Subm}$ assigns to a submersion $\pi\colon X\to Y$ the set of fiberwise Riemannian metrics, i.e., an inner product on the vertical tangent bundle. A ${\sf Riem}$-structure on $X$ is exactly a Riemannian metric on $X$. \label{ex:Riem}\end{ex}

\begin{ex} The sheaf $C^\infty$ on the site ${\sf Subm}$ assigns to a submersion $\pi\colon X\to Y$ smooth function on the total space, viewed as the set of smoothly varying smooth functions on the fibers. \label{ex:Cinfty}\end{ex}

\begin{defn}
\emph{A $\mathcal{G}$-oriented quantization} of $0|\delta$-EFTs is an assignment to each $\mathcal{G}$-oriented proper submersion $p\colon X\to Y$ a map
$$
p_!(g)\colon 0|\delta\EFT^\bullet_{\rm pol}(X)\to 0|\delta\EFT^{\bullet-n}_{\rm pol}(Y)
$$
where the $n$ is the fiber dimension of the proper submersion $p$. The map $p_!$ is called \emph{quantization along $p$}. \label{def:quant}\end{defn} 

\begin{rmk} Construction of quantization procedures for $d|\delta$-dimensional field theories will likely involve sheaves $\mathcal{G}$ on ${\sf Subm}$ of $d$-categories. We have restricted to sheaves of sets in this paper since our examples have~$d=0$.  \end{rmk}

\begin{defn} Given a $\mathcal{G}$-oriented quantization, the \emph{sigma model partition function}, denoted $Z_X^{0|\delta}(g)$ is the image of $1\in 0|\delta\EFT^0_{\rm pol}(X)$ under quantization along $X\to \pt$. \end{defn}

\begin{rmk} The moduli space of super Euclidean geometries on $\R^{0|\delta}$ is a single point (albeit with nontrivial automorphisms), so a partition function of a $0|\delta$-dimensional field theory will be a function on this point, i.e., just a \emph{number}. In higher dimensions we expect an honest function on an interesting moduli space. \end{rmk}

\begin{defn} For a fixed $X$, define a sheaf $\mathcal{G}_X$ on manifolds by 
$$\mathcal{G}_X(Y):=\mathcal{G}(p\colon X\times Y\to Y),$$ 
where $p$ is the projection. The sheaf $\mathcal{G}_X$ is the \emph{smooth space of $\mathcal{G}$-geometries on $X$}. 
\end{defn}

\begin{ex}\label{ex:0|1int} For $0|1$-EFTs, ordinary integration of differential forms provides a quantization for ${\sf Or}$-oriented submersions. Furthermore, when $X$ is an ordinary manifold all field theories are renormalizable. However, the partition function is not so interesting,
$$
Z_X^{0|1}(g)=0,
$$ 
unless $X$ is a 0-manifold in which case $Z_X^{0|1}$ computes a weighted, signed cardinality of~$X$. Although it is usually trivial, $Z_X^{0|1}$ is an additive, multiplicative topological invariant of oriented manifolds. \end{ex}

\begin{ex} The main example in this paper occurs when $\delta=2$. The existence of quantization will rely on some constructions in supergeometry. 
\begin{thm}\label{02push} Let $\mathcal{G}={\sf Riem}\times C^\infty$ be the sheaf parametrizing metrics and smooth functions on the fibers of submersions. The $0|2$-sigma model defines a $\mathcal{G}$-oriented quantization 
\beq
p_!(g) \colon 0|2\EFT_{\rm pol}^\bullet(X) \to 0|2\EFT_{\rm pol}^{\bullet-n}(Y)
\eeq
as in  Equation~\ref{schematic}. 
\end{thm}
\end{ex}

\subsection{Partition functions and concordance}\label{sec:partfunconc}

In this section we examine how quantization and partition functions interact with concordance. We will need the following general lemma. 

\begin{lem}\label{lem:conc} Morphisms of presheaves preserve concordance classes. \label{quantconc1}\end{lem}

\begin{proof} Let $R\colon  \mathcal{E}\to \mathcal{E}'$ be a morphism of presheaves. By naturality, we have commutative diagrams,

$$
\begin{tikzpicture}[>=latex]
\node (A) at (0,0) {$\mathcal{E}(X\times \R)$};
\node (B) at (4,0) {$ \mathcal{E}'(X\times \R)$};
\node (C) at (0,-1.5) {$\mathcal{E}(X\times (\pm 1,\pm \infty))$};
\node (D) at (4,-1.5) {$\mathcal{E}'(X\times (\pm 1,\pm \infty))$};
\draw[->] (A) to node [above=1pt] {$R$} (B);
\draw[->] (A) to node [left=1pt] {$i^*_\pm$} (C);
\draw[->] (B) to node [right=1pt] {$i^*_\pm$} (D);
\draw[->] (C) to node [below=1pt] {$R$} (D);
\end{tikzpicture}
\quad 
\begin{tikzpicture}[>=latex]
\node (A) at (0,0) {$\mathcal{E}(X)$};
\node (B) at (4,0) {$ \mathcal{E}'(X)$};
\node (C) at (0,-1.5) {$\mathcal{E}(X\times (\pm 1,\pm \infty))$};
\node (D) at (4,-1.5) {$\mathcal{E}'(X\times (\pm 1,\pm \infty))$};
\draw[->] (A) to node [above=1pt] {$R$} (B);
\draw[->] (A) to node [left=1pt] {$\pi^*_\pm$} (C);
\draw[->] (B) to node [right=1pt] {$\pi^*_\pm$} (D);
\draw[->] (C) to node [below=1pt] {$R$} (D);
\end{tikzpicture}
$$
so that a concordance $\tilde{E}$ between sections $E_+,E_-\in \mathcal{E}(X)$ maps to a concordance $R(\tilde{E})$ between section $R(E_+)$ and $R(E_-)$ in $\mathcal{E}'(X)$. \ep

\begin{lem} Concordant geometries produce concordant partition functions, i.e., $[g]=[g']\in \mathcal{G}_X[\pt]$ implies that $[Z_X(g)]=[Z_X({g'})]\in d|\delta\EFT[\pt]$. \label{lem:conc0} \end{lem}

\bp If we restrict Definition \ref{def:quant} to the submersions $p\colon X\times Y\to Y$, we can consider quantization as a morphism of presheaves in the variable $Y$,
$$
p_!\colon \mathcal{G}_X(Y)\times 0|\delta\widetilde{\EFT}(X\times Y)\to 0|\delta\EFT(Y).
$$ 
Setting $Y=\pt$ and applying Lemma \ref{quantconc1} to the image of $(g,1)$ under the morphism
\beq\begin{array}{ccccc}
\mathcal{G}_X(\pt)\times 0|\delta\EFT_{\rm pol}^0(X)&\stackrel{Q}{\longrightarrow}& 0|\delta\EFT_{\rm pol}^{-n}(\pt)\subset \R, \\
\downin && \downin  \\
(g,1) & \mapsto & Z_X^{0|\delta}(g) 
\end{array}\nonumber
\eeq
the result follows. 
\ep

\begin{prop} If $[g]=[g']$, there is an equality of partition functions $Z_X(g)=Z_X(g')$. \label{partfunprop} \end{prop}

\begin{proof} 
The Proposition will follow from a result proved in Section \ref{sec:delta}. 

\begin{prop} For $\delta>0$, elements of $0|\delta\EFT_{\rm pol}^k(\pt)\subset \R$ are concordant if and only if they are equal as functions on $\SM(\R^{0|\delta},\pt)\cong \pt$. \label{cor:concpoint} \end{prop}

To finish the proof of Proposition \ref{partfunprop}, we use the above result to identify a function on $\SM(\R^{0|\delta},\pt)$ with the concordance class it represents, finding that in particular $Z_X^{0|\delta}(g)=[Z_X^{0|\delta}(g)]=[Z_X^{0|\delta}(g')]=Z_X^{0|\delta}(g')$ as real numbers, proving Proposition \ref{partfunprop}. \ep

\begin{prop} If a quantization of $0|\delta$-EFTs uses the sheaf $\mathcal{G}={\sf Riem}\times C^\infty$ defined in Examples \ref{ex:Riem} and \ref{ex:Cinfty}, then $[Z_X^{0|\delta}(g)]$ is independent of $g$, i.e., independent of the choice of metric and smooth function on $X$.\label{prop:indep}\end{prop}

\bp This statement basically amounts to the contractibility of the space of metrics and smooth functions. We will construct concordances explicitly. First, let $g_\lambda$ be a smooth 1-parameter family of metrics connecting $g_0$ and $g_1$ that is constant on $g_1$ for $\lambda\ge 1$ and constant on $g_0$ for $\lambda\le -1$. Put the metric $\tilde{g}:=g_\lambda\otimes d\lambda^2$ on $X\times\R$, where $\lambda$ is identified with a global coordinate on $\R$. Similarly, let $h_\lambda$ be a 1-parameter family of functions on~$X$ connecting $h_0$ and $h_1$, which we can then promote to a function $\tilde{h}$ on $X\times \R$, $\tilde{h}(x,\lambda)=h_\lambda(x)$. Together, this shows that $\mathcal{G}_X[\pt]=\pt$, and the Proposition follows. \ep
 
 \begin{proof}[Proof of Corollary \ref{thm3}] By the previous Proposition and Proposition \ref{partfunprop}, the partition function $Z_X(g,h)$ equals the number representing its concordance class, denoted $[Z_X(g,h)]$, which is independent of both $g$ and $h$. 
 \ep
 
\subsection{Notation and conventions}\label{sec:not}
We write ${\sf SM}$ for the category of supermanifolds, and refer the reader to \cite{strings1,STsuper} for preliminaries. To be very brief, objects in this category are locally ringed spaces, $M^{n|m}=(|M|^n,C^\infty)$, where $C^\infty$ is a sheaf of real superalgebras locally isomorphic to $C^\infty(\R^n)\otimes \Lambda^\bullet(\R^m)^*$. We write $|M|$ for the smooth $n$-manifold $(|M|,C^\infty/{\rm nilpotents})$, called the \emph{reduced manifold} of $M$. Since smooth manifolds admit partitions of unity, morphisms of supermanifolds are determined by the induced map on global sections of the sheaf $C^\infty$, whence the slogan ``supermanifolds are affine." We will use this fact without comment throughout. 

Let ${\sf SM}(M,N)$ denote the {\it set} of maps between supermanifolds $M$ and $N$, and $\underline{\sf SM}(M,N)$ the inner hom, i.e.,  the functor 
$$
\underline{\sf SM}(M,N)\colon  {\sf SM}^{\op}\to {\sf SET}, \quad S\mapsto {\sf SM}(S\times M,N).
$$
Similarly, we define $\underline{\sf Diff}(M)$ as the functor
\beq
\underline{\rm Diff}(M)(S)=\left\{ 
\begin{tikzpicture}[baseline=(basepoint)];
\node (A) at (0,0) {$S\times M$};
\node (B) at (3,0) {$S\times M$};
\node (C) at (1.5,-1.5) {$S$};
\node (D) at (1.5,-.6) {$\#$};
\draw[->] (A) to node [above=1pt] {$\cong$} (B);
\draw[->] (A) to (C);
\draw[->] (B) to (C);
\path (0,-.75) coordinate (basepoint);
\end{tikzpicture}\right\}.
\eeq
The above are examples of \emph{functors of points}, which may not be representable as supermanifolds meaning there may not exist a natural isomorphism with a functor
$$
\underline{Y}\colon {\sf SM}^{\op}\to {\sf SET},\quad S\mapsto {\sf SM}(S,Y),
$$ 
where $Y$ is a supermanifold. Still, much of supermanifold theory utilizes the functor of points rather than the supermanifold itself, and a surprising amount can be done with nonrepresentable presheaves on super manifolds, which we shall call \emph{generalized supermanifolds}.\footnote{A better-behaved category consists of \emph{sheaves} on supermanifolds. However, for the purposes of this paper we will stick to the somewhat simpler category of presheaves.}

Even if a generalized supermanifold is representable, whenever we refer to a point $\Phi$ of~$M$, we will implicitly mean a map $\Phi\colon  S\to M$---although the ordinary points of a supermanifold tell us very little (namely, $|M|$) the $S$-points of $M$ tell us everything by the usual Yoneda argument. For example, in Appendix \ref{functions} we explain how functions on a supermanifold are determined by their values at $S$-points, which immediately leads us to the correct notion of functions on generalized supermanifolds. 

We will frequently use the parity reversal functor $\pi$. It has a few incarnations:
\begin{enumerate} 
\item for $A$ a (commutative) superalgebra, $\pi\colon {\sf Mod}_A\to {\sf Mod}_A$ takes a (left or right) $A$-module to the parity reversed (left or right) module; 
\item $\pi\colon {\sf SVBund}\to {\sf SVBund}$ takes a super vector bundle over a supermanifold to the parity reversed bundle; and
\item $\pi\colon {\sf SVBund}\to {\sf SM}$ takes a super vector bundle to the total space of the parity reversed bundle.
\end{enumerate}
When these distinctions matter we will be explicit. 

Throughout, unless stated otherwise, $X$ is assumed to be an ordinary closed manifold, which we will frequently view as a supermanifold, i.e., as its image under the embedding of manifolds in supermanifolds.

\subsection{Outline of the paper}

The next section is the technical heart of the paper, where we define $0|\delta$-EFTs and the homotopy-invariant functor from manifolds to graded algebras gotten by taking concordance classes of field theories. In Section \ref{sec:push} we focus attention on quantization of $0|2$-Euclidean field theories via the Gaussian measure determined by the classical $0|2$-sigma model. In Section \ref{sec:CGB}, we combine the previous results to supply the details in our proof of the Chern-Gauss-Bonnet theorem. 

\subsection{Acknowledgements} 

It is a pleasure to thank Dmitri Pavlov and Stephan Stolz for many useful suggestions, and my advisor Peter Teichner for his insight and support.

\section{$0|\delta$-EFTs and Concordance}\label{sec:cocycle}

In this section we provide an explicit description of $0|\delta\EFT_{\rm pol}^\bullet(X)$ through functor-of-points computations in Sections \ref{sec:funcs} and \ref{sec:grpact}. Similar computations are carried out by Kochan and {\v S}evera in \cite{gorms}. Then we give an algebraic characterization of when $0|\delta$-Euclidean field theories are concordant in Section \ref{sec:delta}, which may be viewed as a generalization of a piece of Stokes Theorem. 

\subsection{$\SM(\R^{0|\delta},X)$ and its functions}\label{sec:funcs}

Using the argument reviewed in Appendix \ref{functions}, we identify an element of $C^\infty(\SM(\R^{0|\delta},X))$ with maps of sets $\SM(\R^{0|\delta},X)(S)\to C^\infty(S)$ natural in $S$. 

We choose coordinates $\{\theta_1,\dots,\theta_\delta\}$ on $\R^{0|\delta}$, which gives isomorphisms natural in $S$, $C^\infty(S\times \R^{0|\delta})\cong C^\infty(S)[\theta_1,\dots,\theta_\delta],$ with $\theta_i$ odd. A map $\Phi$ of supermanifolds is determined by a map $\Phi^*\colon  C^\infty X\to C^\infty(S\times \R^{0|\delta})$ of superalgebras. We can  express $\Phi^*$ in terms of its Taylor components, 
$$
\Phi^*=f+\sum_I \phi_I\theta_I , 
$$
where $I=\{i_1<\dots < i_k\}$ is a nonempty increasing subset of $\{1,\dots,\delta\}$, $\theta_I=\theta_{i_1} \cdots \theta_{i_k}$, and $f,\phi_I\colon  C^\infty(X)\to C^\infty(S)$ are linear maps with restrictions that make $\Phi^*$ an algebra homomorphism. Notice that $f$ induces a map of supermanifolds $S\times \pt\to X$.

Given any $x\in C^\infty X$, we define a function also denoted $x\in C^\infty(\SM(\R^{0|\delta},X))$ whose value at an $S$-point $\Phi$ is $x(\Phi)=f(x).$ For a map $s\colon  S'\to S$, the value of the function $x$ at the $S'$-point is $x(s\circ \Phi)=(s\circ f)(x)\in C^\infty(S')$, so that $x$ is indeed natural in $S$ and therefore defines an honest function on $\SM(\R^{0|\delta},X)$. This gives an inclusion of algebras 
\beq
C^\infty X\hookrightarrow C^\infty(\SM(\R^{0|\delta},X)).\label{eq:funinclude}
\eeq
Other examples of functions are denoted by $d_Ix$ for $x\in C^\infty X$, whose value at an $S$-point is defined as $(d_Ix)(\Phi):=\phi_I(x).$ We note that the dilation action of $\R^\times$ on $\R^{0|\delta}$ is through a dilation action on the coordinates $\{\theta_i\}$, and induces an action on $d_Ix$ by $r^{|I|}$ for $r\in \R^\times$. Hence, $d_Ix$ has polynomial degree $|I|$ in the sense of Definition \ref{def:pol}.

We can form arbitrary smooth functions in the variables $d_Ix$, in the sense that if $\{x^j\}$ are local coordinates on $X$, $\{d_Ix^j\}$ are local coordinates on $\underline{\sf SM}(\R^{0|\delta},X)$. By the usual sheaf property for functions on a supermanifold, this proves the following. 

\begin{prop} The algebra $C^\infty(\underline{\sf SM}(\R^{0|\delta},X))$ is generated by smooth functions in $d_Ix$ for $x\in C^\infty X$ and $I$ varying over all multi-indices $\{i_1,\dots,i_k\}$. The super algebra $C^\infty_{\rm pol}(\underline{\sf SM}(\R^{0|\delta},X))$ is  freely generated by the variables $d_Ix$ for $x\in C^\infty X$. \label{funs} \end{prop}

\begin{ex}  \label{ex:piTX} We now unravel the above computations in the case of differential forms. This example can be found in various guises in many places, for example \cite{strings1,kont,HKST}. We compute the $S$-points,
$$
\SM(\R^{0|1},X)(S)\cong\{ \Phi\colon  S\times\R^{0|1}\to X\}\cong \{\Phi^*\colon  C^\infty X \to C^\infty(S)\otimes C^\infty(\R^{0|1})\}.
$$
Choosing a coordinate $\theta$ on $\R^{0|1}$ gives a decomposition, 
$$
C^\infty(S)\otimes C^\infty(\R^{0|1})\cong C^\infty S\oplus C^\infty S\cdot \theta
$$
so we may express $\Phi$ in terms of the Taylor components, $\Phi^*=f+\phi\theta.$ Enforcing the condition that $\Phi^*$ be an algebra homomorphism we find $f\colon  C^\infty X\to C^\infty S$ is a grading-preserving algebra homomorphism and $\phi\colon  C^\infty X\to C^\infty S$ is a grading-reversing map that is an odd derivation with respect to $f$, 
$$
\phi(ab)=\phi(a)f(b)+(-1)^{p(a)}f(a)\phi(b),\quad a,b\in C^\infty(X).
$$ 
But this is the standard description \cite{strings1} of $\pi TX$ in terms of its $S$-points, which recovers the isomorphism $\SM(\R^{0|1},X)\cong \pi TX. $ We can define functions on this space for any $x\in C^\infty(X)$ by assigning their values on $S$-points as
$$
x(\Phi):=f(x), \quad dx(\Phi):=\phi(x).
$$
These are the zero- and one-forms in $\Omega^\bullet(X)\subset C^\infty(\pi TX) $, respectively. For $X$ an ordinary manifold, these generate $C^\infty(\underline{\sf SM}(\R^{0|1},X))$ as an algebra. The dilation action of $\R_{>0}$ on~$\R^{0|1}$ gives an $\R_{>0}$-action on $C^\infty(\underline{\sf SM}(\R^{0|1},X))$ whose eigenspaces are (homogeneous) polynomial functions in the sense of Definition \ref{def:pol}, indexed by $\N$; the $k$th eigenspace for $k\in \N$ consists of degree $k$ differential forms, $\Omega^k(X)$. 
\begin{rmk}
Following the remark on page 74 of \cite{strings1}, we can describe differential forms on supermanifolds in terms of functions on the odd tangent bundle.
\begin{prop}
Let $M$ be a supermanifold. Then there is an isomorphism of sheaves,
$$
\Omega^\bullet(M)\cong C^\infty_{\rm pol}(\underline{\sf SM}(\R^{0|1},M)),
$$
between polynomial functions on $\pi TM$ and differential forms on $M$. 
\end{prop}
\noindent We emphasize that the polynomial condition is essential for supermanifolds~$M$; for example, the algebra of smooth functions on (the total space of) $\pi T(\R^{0|1})\cong \R^{1|1}$ is~$C^\infty(\R)[\theta]$, which is much larger than the algebra of differential forms, $\Omega^\bullet(\R^{0|1})\cong \R[\theta,d\theta]$. However, for ordinary manifolds $X$, we have a natural isomorphism, $C^\infty_{\rm pol}(\underline{\sf SM}(\R^{0|1},X))\cong C^\infty(\underline{\sf SM}(\R^{0|1},X)).$
\end{rmk}
\end{ex}

We now consider our main example, when $\delta=2$, in more detail. 

\begin{ex} 
Consider the $S$-points with a choice of coordinate, 
$$
\SM(\R^{0|2},X)(S)\cong\{ \Phi^*\colon C^\infty(X)\to C^\infty(S)[\theta_1,\theta_2]\},
$$
which allows us to write Taylor components
$$
\Phi^*=f+\phi_1\theta_1+\phi_2\theta_2+E\theta_1\theta_2,
$$
where $\phi_i\colon C^\infty X\to (C^\infty S)^{\rm odd}$ and $f,E\colon C^\infty X\to (C^\infty S)^{\rm even}$. A computation shows that~$\Phi^*$ is an algebra homomorphism if and only if
\beq
f(ab)&=&f(a)f(b) \nonumber\\
\phi_i(ab)&=&\phi_i(a)f(b)-f(a)\phi_i(b) \quad i=1,2 \label{r02Spt}\\
E(ab)&=&E(a)f(b)+f(a)E(b)+\phi_1(a)\phi_2(b)+\phi_1(b)\phi_2(a).\nonumber
\eeq
so $f$ is an algebra homomorphism, $\phi_1$ and $\phi_2$ are odd derivations with respect to $f$, and $E$ satisfies the above quadratic identity.

The polynomial functions on $\SM(\R^{0|2},X)$ are generated as an algebra by
\beq
x(\Phi)=f(x),\ \ (d_1x)(\Phi)=\phi_1(x), \ \ (d_2x)(\Phi)=\phi_2(x), \ \ (d_2d_1x)(\Phi)=E(x).\label{eq:02funs}
\eeq
The function $x$ has polynomial grading $0$, $d_1x$ and $d_2x$ have polynomial grading $+1$, and $d_2d_1x$ has grading~$+2$. We remark that unlike $d_1x$ and $d_2x$, the element $d_2d_1x$ is not nilpotent, so polynomials in the above variables are a strict subset of $C^\infty(\SM(\R^{0|2},X))$ when ${\rm dim}(X)>0$. 

\end{ex}

\subsection{Group actions on $\SM(\R^{0|\delta},X)$}\label{sec:grpact}

Let $\A\in \Diff(\R^{0|\delta})(S)$ and $\Phi\in \SM(\R^{0|\delta},X)(S)$, i.e., $\A\colon  S \times \R^{0|\delta}\stackrel{\cong}{\to} S\times \R^{0|\delta}$ and $\Phi\colon  S\times \R^{0|\delta}\to X,$ where $\A$ is a map of bundles over $S$. By restricting $\A\in \Euc(\R^{0|\delta})(S)\subset \Diff(\R^{0|\delta})(S)$, we define an action on $S$-points as
$$
\begin{tikzpicture}[>=latex]
\node (A) at (0,0) {$ \SM(\R^{0|\delta},X)(S)\times \Euc(\R^{0|\delta})(S) $};
\node (B) at (6,0) {$\SM(\R^{0|\delta},X)(S)$};
\node (C) at (0,-.6) {$\Phi,\A$};
\node (D) at (6,-.6) {$\Phi\circ \A.$};
\node (E) at (2.4,0) {\empty};
\node (F) at (2.4,-.6) {\empty};
\draw[->] (E) to  (B);
\draw[|->] (F) to (D);
\end{tikzpicture}
$$
Let $\euc(\R^{0|\delta})$ denote the Lie algebra of $\Euc(\R^{0|\delta})$. The infinitesimal action of odd translations leads to odd vector fields on $\SM(\R^{0|\delta},X)$ that raise the polynomial degree of functions by $1$. For a chosen basis of $\R^{0|\delta}$, let $D_i$ denote odd vector field associated to the action by the $i$th basis vector. We have a homomorphism of Lie algebras from $\R^{0|\delta}$ into vector fields on $\SM(\R^{0|\delta},X)$, and so we also get an induced homomorphism of universal enveloping algebras from $\Sym(\R^{0|\delta})$ into differential operators on $\SM(\R^{0|\delta},X)$; here we are using that $\R^{0|\delta}$ is a superabelian Lie algebra, so its universal enveloping algebra is the (graded) symmetric algebra on its Lie algebra. We observe that $D_iD_j=-D_jD_i$. Let $D_I$ denote the differential operator obtained from the composition $D_{i_1}\cdots D_{i_k}$ for a given ordered set $I=\{i_1,\dots, i_k\}$. The following characterizes the action of $D_I$ on~$C^\infty(\SM(\R^{0|\delta},X))$. 

\begin{lem}\label{funcs}
Let $x\in C^\infty(X)\subset C^\infty(\SM(\R^{0|\delta},X))$. Then $D_Ix=d_Ix$, where the left hand side is the action of the differential operator $D_I$ on the function $x$, and the right hand side is the function $d_Ix$ defined in the previous section. \end{lem}

\begin{proof} 
Consider the action of $D_i$ on $d_Ix\in C^\infty(\SM(\R^{0|\delta},X))$ in terms of the functor of points. At an $S$-point $\Phi$, $D_i$ acts as
$$
D_i\colon  C^\infty(X)\stackrel{\Phi^*}{\to} C^\infty(S)\otimes C^\infty(\R^{0|\delta})\stackrel{{\rm id}\otimes \partial_{\theta_i}}{\longrightarrow} C^\infty(S)\otimes C^\infty(\R^{0|\delta}). 
$$
where $\partial_{\theta_i}$ is the vector field on $\R^{0|\delta}$ associated with infinitesimal translations in the $\theta_i$-direction. If we express an $S$-point in terms of its Taylor expansion and consider the action of $D_i$ on the function $d_Ix$, we find
$$
\sum \phi_I\theta_I\stackrel{D_i}{\mapsto} \sum \phi_I \partial_{\theta_i}\theta_I\stackrel{d_Ix}{\mapsto} \phi_{i\bigcup I}(x)
$$
so that $D_i(d_Ix)=d_{i\cup I} x$. Given $I=\{i_1,\dots,i_k\}$, we can iterate the above action on the function $x\in C^\infty(\SM(\R^{0|\delta},X))$, finding
$$
D_I(x)(\Phi)=(d_Ix)(\Phi)=\phi_I(x),
$$
as claimed. \end{proof}

\begin{notation} Following the previous lemma, we use $d_Ix$ to denote both the function $d_Ix\in C^\infty(\SM(\R^{0|\delta},X))$ and an operator $d_I:=D_I$ acting on the functions. In particular, the action of the $i$th basis vector of $\R^{0|\delta}$ is denoted by $d_i$, and these operators have polynomial grading~$+1$. Let $\Delta$ denote the composition~$d_\delta\cdots d_1$; it has polynomial degree $+\delta$, meaning is sends functions of polynomial degree~$k$ to polynomial degree~$\delta+k$. 
\end{notation}

The previous lemma together with our characterization of $C^\infty(\SM(\R^{0|\delta},X))$ makes it possible to describe the action by the Euclidean group. The action of odd translations $\partial_{\theta_i}\in\R^{0|\delta}$ is determined by the formula $\partial_{\theta_i}\cdot (d_Ix)=(d_i d_I) x$ together with the relations $d_id_j=-d_jd_i$. To compute the action of~$O(\delta)$, observe that its action on the Lie superalgebra~$\R^{0|\delta}$ naturally extends to one on the universal enveloping algebra of~$\R^{0|\delta}$, $\Sym(\R^{0|\delta})$. Hence~$A\in O(\delta)$ acts on a function $d_Ix$ by $A(d_Ix)=(Ad_I)x.$ Concretely, this is the standard action of~$O(\delta)$ on $\Sym(\R^{0|\delta})$, which (ignoring gradings) is an exterior algebra on $\R^\delta$. 

The other action we need to understand is defined at an $S$-point by
$$
\begin{tikzpicture}[>=latex]
\node (A) at (0,0) {$ \SM(\R^{0|\delta},X)(S)\times \SM(\R^{0|\delta},\Diff(X))(S) $};
\node (B) at (6,0) {$\SM(\R^{0|\delta},X)(S)$};
\node (C) at (0,-.6) {$\Phi,\G$};
\node (D) at (6,-.6) {$\Phi\cdot \G\phantom{=\Phi\circ fi\A}$};
\node (E) at (3,0) {\empty};
\node (F) at (3,-.6) {\empty};
\draw[->] (E) to  (B);
\draw[|->] (F) to (D);
\end{tikzpicture}
$$
where we view $\G$ as an automorphism of the trivial bundle over $S\times \R^{0|\delta}$ with fiber $X$, $S\times \R^{0|\delta}\times X\stackrel{\G}{\to} S\times \R^{0|\delta}\times X,$ and can turn $\Phi$ into a section of this bundle via 
$$
\id\times \Phi\in \Gamma(S\times \R^{0|\delta},S\times \R^{0|\delta}\times X),
$$
and finally we define $\Phi\cdot \G$ as the composition
$$
S\times \R^{0|\delta}\stackrel{\id\times \Phi}{\longrightarrow} S\times \R^{0|\delta}\times X\stackrel{\G}{\to} S\times \R^{0|\delta}\times X\stackrel{p}{\to} X
$$
where $p$ is projection. We can consider the corresponding infinitesimal action at the level of the Lie algebra, $\underline{\sf SM}(\R^{0|\delta},\Gamma(TX))$. For our purposes we need only consider the action by elements denoted $\mathcal{L}_v,I_w\in \underline{\sf SM}(\R^{0|\delta},\Gamma(TX))$ defined for $v,w\in \Gamma(TX)$; these operators have Taylor components at an $S$-point of the form
$$
\mathcal{L}_v:=v,\quad I_w:=w\theta_1\dots\theta_\delta.
$$
It is straightforward to check that $\mathcal{L}_v$ has polynomial degree~0 and $I_w$ has polynomial degree~$-\delta$. As suggested by the notation, $\mathcal{L}_v$ acts by the Lie derivative, and the relevance of the operator $I_w$ comes from the formula
\beq
[d_\delta,\dots,[d_2,[d_1,I_w]]\dots]=\mathcal{L}_w,\label{eq:cartan}
\eeq
generalizing the Cartan formula. To explain the above equality, we consider the action of the left side on the function $d_Jx$.  Expanding the expression $[d_\delta,\dots,[d_2,[d_1,I_w]\dots]$, we get a sum of terms of the form $d_K I_w d_L$ for $K\cup L\cong \{ 1,\dots,n\}$ (though not necessarily as ordered sets). If $i\in J$ and $i\in L$, then $d_K I_w d_L (d_Jx)=d_K I_w(0)=0,$ using the fact that $d_i^2=0$ for all $i$. As usual, the value of $d_Jx$ at an $S$-point $\sum \phi_I\theta_I$ is $\phi_J(x)$, and we can understand the action of the operator $I_w$ on $d_Ld_Jx$ by precomposing with the action on the $S$-point. If $i\notin J$ and $i\notin L$, then $I_w(d_L(d_Jx))=0$, using the definition of $I_w$ and the fact that $\theta_i^2=0$ for all~$i$. From this it follows that any nontrivial action of $[d_\delta,\dots,[d_2,[d_1,I_w]]$ on $d_Jx$ arises from terms where $L\cup J\cong \{1,\dots,n\}$ and hence $K\cong J$, where again these isomorphisms may not preserve the ordering of these sets. In this case we compute
$$
d_KI_wd_L(d_Jx)=d_KI_w (\Delta x)=d_K(wx)=d_J(wx)=\mathcal{L}_w d_Jx,
$$
where there are possible signs we have suppressed, owing to the non-ordered isomorphisms $K\cong J$ and $L\cup J\cong \{1,\dots ,n\}$. However, these sign ambiguities exactly cancel, so that 
\beq
[d_\delta,\dots,[d_2,[d_1,I_w]]\dots]d_Jx=\mathcal{L}_w d_Jx
\eeq
Since $[d_\delta,\dots,[d_2,[d_1,I_w]]\dots]$ is a derivation (being an iterated Lie bracket of derivations) the above characterizes its action on functions via the Leibniz rule and we have proved formula~(\ref{eq:cartan}). 

\begin{rmk} The action by the Euclidean group is natural in $X$ since a map $X\to Y$ induces a $\Euc(\R^{0|\delta})$-equivariant map $\underline{\sf SM}(\R^{0|\delta},X)\to \underline{\sf SM}(\R^{0|\delta},Y)$, i.e., a morphism of Lie groupoids~$\underline{\sf SM}(\R^{0|\delta},X)\sq \Euc(\R^{0|\delta})\to \underline{\sf SM}(\R^{0|\delta},Y)\sq \Euc(\R^{0|\delta})$. This generalizes the usual naturality of the de~Rham $d$. In particular, $\Delta$ acts naturally, which will be important in the next subsection. \end{rmk}

We now explain how the above actions give rise to familiar algebraic structures on differential forms when $\delta=1$, and then we explain in detail the situation for $\delta=2$. 

\begin{ex}\label{0|1deRham} We have that $C^\infty(\underline{\sf SM}(\R^{0|1},X))\cong \Omega^\bullet(X)$; the Euclidean group in this example acts through an $\R^{0|1}$-action and a $O(1)\cong\Z/2$-action. The infinitesimal generator of $\R^{0|1}$ acts by the de~Rham $d$, and the $\Z/2$-action is by $+1$ on even forms and~$-1$ on odd forms. We get an infinitesimal action from the (infinite dimensional) Lie algebra
$$
{\rm Lie}(\underline{\sf SM}(\R^{0|1},\underline{\rm Diff}(X)))\cong \underline{\sf SM}(\R^{0|1},\Gamma(TX))\cong \Gamma(TX)\oplus \pi \Gamma(TX)
$$ 
where in the above we view $\Gamma(TX)$ as a generalized manifold whose functor of points is $C^\infty(S,\Gamma(TX))$, i.e., smooth functions with values in $\Gamma(TX)$. Then the $S$-points of the inner hom $\underline{\sf SM}(\R^{0|1},\Gamma(TX))$ can be identified with $C^\infty(S,\Gamma(TX))[\theta]$; Taylor expanding in $\theta$ we get a term in $\Gamma(TX)$ and a term in $\pi \Gamma(TX)\cong \Gamma(TX)\otimes \R^{0|1}$, which gives the second isomorphism in the above displayed equation. We claim that $v\in \Gamma(TX)$ acts by the Lie derivative, $\mathcal{L}_v$, and $\psi\in \pi \Gamma(TX)$ by interior multiplication, $\iota_{\psi}$, and these change $\N$-degrees by~0 and~$-1$, respectively. We see this by computing the composition that defines the action,
$$
C^\infty(S)[\theta]\otimes C^\infty X\stackrel{\mathcal{G}^*}{\to} C^\infty S[\theta]\otimes C^\infty X\stackrel{\Phi^*}{\to} C^\infty(S)[\theta]
$$
where $\mathcal{G}^*=v+\psi\theta$, $(v,\psi)\in \Gamma(TX)\oplus \pi \Gamma(TX)$, and $\Phi=f+\phi\theta$. Then we find on functions
$$
(\mathcal{G}^* x)(\Phi)=f(vx)=(\mathcal{L}_vx)(\Phi)\quad (\mathcal{G}^*dx)(\Phi)=\phi(vx)+ f(\psi x)=(\mathcal{L}_vdx)(\Phi)+(\iota_\psi dx)(\Phi),
$$
which follows from the action of $\mathcal{G}$ on the $S$-point, $f+\phi\theta\stackrel{\mathcal{G}^*}{\mapsto}  f\circ v+(f\circ \psi)\theta +(\phi \circ v)\theta$. Since $\mathcal{G}^*\in {\rm Lie}(\underline{\sf SM}(\R^{0|1},\Diff(X)))$ acts by derivations, the above formulas determine the action uniquely on $C^\infty(\underline{\sf SM}(\R^{0|1},X))\cong \Omega^\bullet(X)$. 
An identical (though simpler) argument as in the proof of Equation \ref{eq:cartan} proves the usual Cartan identity, $[d,\iota_V]=\mathcal{L}_V.$ 
\end{ex}

\begin{ex} Functions on $\underline{\sf SM}(\R^{0|2},X)$ are generated by the monomials in Equation~\ref{eq:02funs}. The odd translations, $\R^{0|2}$, act in the predictable way that was described before, using that $d_1d_2=-d_2d_1$. For $x\in C^\infty X$, the rotations $O(2)$ act via the usual 2-dimensional representation on the span of $d_1x,d_2x$; act trivially on $x$; and act through the determinant homomorphism on $d_2d_1x$. 

Next we wish to understand the action by
$$
\underline{\sf SM}(\R^{0|2},\Gamma(TX))\cong  \Gamma(TX)\oplus \pi \Gamma(TX)\oplus  \pi \Gamma(TX)\oplus  \Gamma(TX).
$$
As in the previous example, we consider the composition
$$
C^\infty (S)[\theta_1,\theta_2]\otimes C^\infty X\stackrel{\mathcal{G}^*}{\to} C^\infty (S)[\theta_1,\theta_2]\otimes C^\infty X\stackrel{\Phi^*}{\to} C^\infty(S)[\theta_1,\theta_2],
$$
where $\mathcal{G}^*=v+\psi_1\theta_1+\psi_2\theta_2+w\theta_1\theta_2,$ $\Phi^*=f+\phi_1\theta_1+\phi_2\theta_2+E\theta_1\theta_2,$ and $v,w\in \Gamma(TX)$, $\psi_1,\psi_2\in \pi \Gamma(TX)$. The action of $v$ is by the Lie derivative, $\mathcal{L}_v$,
$$
(\mathcal{L}_vx)(\Phi)=f(vx), \ (\mathcal{L}_vd_ix)(\Phi)=\phi_i(vx), \ (\mathcal{L}_vd_2d_1x)(\Phi)=E(vx). 
$$ 
Most of the action of $\psi_1$ and $\psi_2$ can be computed by considering inclusions $\R^{0|1}\hookrightarrow \R^{0|2}$: when restricting to the subspaces generated by $\{ x, d_1x\}$ or $\{x, d_2x\}$, we get copies of the Cartan algebra. Explicitly, we denote the action of $\psi_i$ by $\iota_{\psi_i}$, respectively and compute
$$
\iota_{\psi_1} x(\Phi)=0, \ \iota_{\psi_1} d_1x(\Phi)=(\mathcal{L}_{\psi_1} x)(\Phi), \ \iota_{\psi_1}d_2x(\Phi)=0, \ \iota_{\psi_1}d_2d_1x(\Phi)=d_2x(\Phi).
$$
Similar formulas hold for the action of ${\psi_2}$. We denote the action of $w$ by $I_w$ and compute
$$
I_w x=0, \ I_w d_1x=0, \ I_w d_2x=0, \ (I_w d_2d_1x)(\Phi)=f(wx)=(\mathcal{L}_wx)(\Phi),
$$
most of which can be deduced by the fact that $I_w$ lowers a function's $\N$-degree by $2$. Finally, following the argument proving (\ref{eq:cartan}) we note the identity $[d_2,[d_1,I_w]]=\mathcal{L}_w.$
\end{ex}

\subsection{Concordance classes of $0|\delta$-EFTs} \label{sec:delta}
The following proposition is the key to computing concordance classes.

\begin{prop} \label{concord} Let $\delta>0$.  Then two twisted $0|\delta$-dimensional Euclidean field theories over $X$ are concordant if and only if they are $\Delta$-cohomologous:
$$
[{\sf E_-}]=[{\sf E_+}] \quad \iff \quad {\sf E_+-E_-}=\Delta e,
$$ 
for $e\in C^\infty(\underline{\sf SM}(\R^{0|\delta},X))$ where $e$ satisfies equivariance properties such that $\Delta e$ is a twisted field theory. An identical statement also holds for ${\sf E}_-$ and ${\sf E}_+$ being twisted \emph{renormalizable} field theories. 
\end{prop}
\begin{rmk} We observe that when $\delta=0$, any two field theories over $X$ are concordant. Hence concordance classes of $0|0$-EFTs over $X$ yield a trivial manifold invariant. In light of this, the above result sets the supersymmetric field theories over~$X$ apart from the non-supersymmetric ones, and provides an algebraic characterization for how concordance classes of field theories encode topological data. \end{rmk}

\begin{proof} All the functions and operators employed below respect the polynomial degree on the algebra $C^\infty_{\rm pol}(\underline{\sf SM}(\R^{0|\delta},X))$, so the argument automatically applies to renormalizable field theories. 

Let $\lambda$ be a coordinate on $\R$. If $E_+-E_-=\Delta e$, then define
$$
\tilde{E}(\lambda):=E_-+\Delta(b(\lambda)\cdot e)\in C^\infty(\SM(\R^{0|\delta},X\times \R))
$$
for $b$ a smooth bump function that is equal to $0$ on $(-\infty,-1]$ and to $+1$ on $[1,\infty)$. The action of $\Euc(\R^{0|\delta})$ on $E_\pm$ is through the 1-dimensional representation determined by the twist, and since $E_+-E_-=\Delta e$, we have that $\Euc(\R^{0|\delta})$ acts on $\Delta e$ through this 1-dimensional representation. 

We clam that $\tilde{E}$ is a field theory of the appropriate twist: since $O(\delta)$ acts trivially on the subspace $C^\infty(M)\subset C^\infty(\SM(\R^{0|\delta},M))$ for any $M$, it acts trivially on $b(\lambda)$. The action of $\R^{0|\delta}\rtimes O(\delta)$ is through algebra automorphisms, and using the fact that the operator $\Delta$ comes from the action of $\R^{0|\delta}<\R^{0|\delta}\rtimes O(\delta)$ we find that the action of the Euclidean group on $\Delta(b(\lambda)e)$ is through the same 1-dimensional representation as the action on~$\Delta e$.  Hence, $\tilde{E}$ is a twisted field theory of the appropriate degree. Examining the various pullbacks, $\tilde{E}$ gives a concordance. 

Now suppose that $\tilde{E}$ is a concordance from $E_+$ to $E_-$, and let $\partial_\lambda$ be a nonvanishing vector field on $\R$ associated to a choice of coordinate $\lambda$. We employ an argument similar to one that proves part of Stokes' Theorem; namely we shall define a linear map 
$$
Q\colon 0|\delta\EFT^\bullet(X\times \R)\to C^\infty(\underline{\sf SM}(\R^{0|\delta},X))
$$
with the property $\Delta Q=i_+^*-i_-^*,$ so that we may take $Q(\tilde{E})=: e$. Let
$$
Q(\tilde{E}):=\int_{-1}^1 i^*_\lambda I_{\partial_\lambda} \tilde{E} d\lambda,
$$
where we view the integral as a $0|\delta\EFT^\bullet(X)$-valued function on $\R$. We compute 
$$
\mathcal{L}_{\partial_\lambda} \tilde{E}=[d_\delta,\dots,[d_2,[d_1,I_{\partial_\lambda}]]\dots]\tilde{E}=\Delta I_{\partial_\lambda}\tilde{E}
$$
where the first equality uses Equation \ref{eq:cartan} and the second expands the Lie brackets into a sum of operators acting on $\tilde{E}$, uses the fact that $d_k\tilde{E}=0$ for all $k$, and observes that the only remaining nonzero term is $\Delta I_{\partial_\lambda}\tilde{E}$. 

We calculate
$$
\Delta Q\tilde{E}=\int_{-1}^1 i^*_\lambda \Delta I_{\partial_\lambda} \tilde{E} d\lambda=\int_{-1}^1 i^*_\lambda \mathcal{L}_{\partial_\lambda} \tilde{E} d\lambda=i^*_+\tilde{E}-i^*_-\tilde{E},
$$
where the first equality is differentiation under the integral together with naturality of $\Delta$, and the last is the fundamental theorem of calculus. Thus, we have shown that $E_+$ and $E_-$ are $\Delta$-cohomologous. \end{proof}

\begin{proof}[Proof of Proposition \ref{cor:concpoint}.] We compute
\beq
0|\delta\EFT^{\ev}(\pt)&\cong& C^\infty(\SM(\R^{0|\delta},\pt)\sq \Euc(\R^{0|\delta}))^{\ev} \nonumber \\
&\cong& C^\infty (\pt\sq \Euc(\R^{0|\delta}))^{\ev}\cong \left(C^\infty(\pt)^{\Euc(\R^{0|\delta})}\right)^{\ev}\cong \R,\nonumber
\eeq
and by a similar computation $0|\delta\EFT^{\odd}(\pt)=\{0\}.$ The action of $\Euc(\R^{0|\delta})$ on $\pt$ is trivial, so the action on functions is also trivial. Hence if $E_+$ and $E_-$ are concordant, we have
$$
E_+-E_-=\Delta e=0 \ \implies \ E_+=E_-.
$$
When $X=\pt$, all field theories over $X$ are renormalizable, so the above computation applies to $0|\delta\EFT_{\rm pol}^\bullet(X)$ as well. 
\end{proof}

\section{Quantization and the $0|2$-Sigma Model}\label{sec:push}

In this section we first describe the space of fields $\SM(\R^{0|2},X)$ in familiar geometric terms and prove Lemma~\ref{geo}. This leads to \emph{component fields}, c.f., \cite{5lectures}. We discuss integration on the supermanifold $\SM(\R^{0|\delta},X)$, and then define the action functional in dimension~$0|2$. We express this action in terms of the component fields, proving Lemma~\ref{action}. Finally, we verify that the action determines a Gaussian measure with the desired properties, proving the first half of Theorem \ref{02push}.

\subsection{Component fields}\label{comp}

We prove Lemma \ref{geo} by giving a bijection on $S$-points,
$$
p^*\pi( TX\oplus TX)(S)\cong \SM(\R^{0|2},X)(S).
$$ 
To define the map, there is some preliminary work to be done. Recall that the ordinary covariant Hessian is a map over $X$
$$
\Hess\colon  TX\otimes TX\to {\sf Diff}^{\le 2}(X)
$$
that takes pairs of tangent vectors and outputs a second order differential operator. We can also define the Hessian on pairs of odd tangent vectors via the isomorphism
$$
\pi TX\otimes \pi TX=(\underline{\R}^{0|1}\otimes TX)\otimes (\underline{\R}^{0|1}\otimes TX)\stackrel{\sigma}{\cong} (\underline{\R}^{0|1}\otimes \underline{\R}^{0|1})\otimes (TX\otimes TX)\cong TX\otimes TX
$$
where $\underline{\R}^{0|1}$ is the trivial odd line over $X$, we use that $\pi TX:=\underline{\R}^{0|1}\otimes TX$, $\R^{0|1}\otimes \R^{0|1}\cong \R$, and $\sigma$ denotes the braiding isomorphism. Precomposing $\Hess$ with the above gives a map of vector bundles over $X$, 
\beq
\Hess\colon  \pi TX \otimes \pi TX\to {\sf Diff}^{\le 2}(X).\label{eq:1}
\eeq
Given an $S$-point $f\colon S\to X$, we can pull back to obtain a map over $S$,
$$
f^*\Hess\colon  (f^*\pi TX)\otimes (f^*\pi TX)\to f^*{\sf Diff}^{\le 2}(X). 
$$
Recall that an $S$-point of $\underline{\sf SM}(\R^{0|2},X)$ is a quadruple $(f,\phi_1,\phi_2,E)$, where
\beq
\Phi^*=f+\phi_1\theta_1+\phi_2\theta_2+ E\theta_1\theta_2, \quad \Phi^*\in {\sf ALG}(C^\infty X,C^\infty S[\theta_1,\theta_2]).\label{eq:02Spt}
\eeq
We can plug $\phi_1$ and $\phi_2$ into the above map and get 
$$
(f^*\Hess)(\phi_1,\phi_2)\in \Gamma(f^*{\sf Diff}^{\le 2}(X)). 
$$
We note $S$-points of ${\sf Diff}^{\le 2}(X)$ are maps of vector spaces $C^\infty X\to C^\infty S$ satisfying some additional conditions. Explicitly, on $X$ there is the evaluation map
$$
\Gamma({\sf Diff}^{\le 2}(X))\otimes_\R C^\infty X\to C^\infty X,
$$
which is a map of sheaves of $C^\infty X$-modules via the left action of $C^\infty X$ on differential operators. Using the map $f^*\colon  C^\infty X\to C^\infty S$, we obtain a map of sheaves of $C^\infty S$-modules 
$$
C^\infty(S)\otimes_{f^*} \Gamma(f^*{\sf Diff}^{\le 2}(X))\otimes_\R C^\infty(X)\to C^\infty(S)\otimes_{f^*} C^\infty(X)\cong C^\infty(S). 
$$
So in particular, given $f^*\Hess(\phi_1,\phi_2)\in \Gamma(f^*{\sf Diff}^{\le 2}(X))$ and $x\in C^\infty X$, we get a function in $C^\infty(S)$. 

\begin{lem} \label{geo1} Let $(f,\phi_1,\phi_2,E)$ be an $S$-point of $\underline{\sf SM}(\R^{0|2},X)$. Then 
$$(f,\phi_1,\phi_2,E-(f^*\Hess)(\phi_1,\phi_2))$$ 
is an $S$-point of $p^*\pi(TX\oplus TX)$. Equivalently, $F:=E-(f^*\Hess)(\phi_1,\phi_2))$ is an even derivation with respect to $f$. 
\end{lem}
\begin{proof}[Proof of Lemma \ref{geo} using Lemma \ref{geo1}.] The map in Lemma \ref{geo1} is natural in $S$ so gives the required map on the functor of points. This map is invertible, implying Lemma \ref{geo}.\end{proof}

\begin{proof}[Proof of Lemma \ref{geo1}.] The proof follows from direct computation. We emphasize our assumption that $X$ is an ordinary manifold, though a similar result holds with some extra signs for a general supermanifold target. 

The Hessian is $C^\infty X$-linear in both vectors, and so is a map of sheaves of $C^\infty X$-modules. We have the formula\footnote{This can be verified directly with the classical formula for the Hessian, $v\otimes w\mapsto vw-\nabla_v w,$ with some care not to introduce extra signs from the braiding isomorphism, $\sigma$.}
$$
\Hess(\phi_1,\phi_2)(ab)=(\Hess(\phi_1,\phi_2)a)\cdot b+a\cdot \Hess(\phi_1,\phi_2)b+\phi_1(a)\cdot \phi_2(b)+ \phi_1(b)\cdot \phi_2(a).
$$
on $X$, and so when we pull back the Hessian to $S$, for $\phi_1,\phi_2\in \Gamma(f^*\pi TX)$, and $a,b\in C^\infty X$ we find
\beq
f^*\Hess(\phi_1,\phi_2)(ab)&=&(f^*\Hess(\phi_1,\phi_2)(a))\cdot f(b)+f(a)\cdot f^* \Hess(\phi_1,\phi_2)(b)\nonumber \\
&& +\phi_1(a)\cdot \phi_2(b)+ \phi_1(b)\cdot \phi_2(a),\nonumber 
\eeq
where both sides are elements of $C^\infty S$. The above argument is the functor of points version that the Hessian---being a tensor---is determined by its value (and well-defined) at points. 

We recall the condition for $E$ to be a component of an $S$-point of $\SM(\R^{0|2},X)$:
$$
E(ab)=E(a)f(b)+f(a)E(b)+\phi_1(a)\phi_2(b)+\phi_1(b)\phi_2(a).
$$
Upon subtracting, $F:=E-f^*\Hess(\phi_1,\phi_2)$ is an even derivation, 
\beq
(E-f^*\Hess(\phi_1,\phi_2))(ab)&=&E(a)f(b)-(f^*\Hess(\phi_1,\phi_2)(a)) f(b)\nonumber \\
&&+f(a)E(b)-f(a) (f^* \Hess(\phi_1,\phi_2)(b))\nonumber \\
&=&F(a)f(b)+f(a)F(b).\nonumber
\eeq
This completes the proof. \end{proof}

\begin{rmk} Lemma \ref{geo} also allows us to say something about spaces of fields for other $d|2$ field theories. Notice
$$
\underline{\sf SM}(\Sigma^{d|2},X)\cong \underline{\sf SM}(\Sigma^{d|0},\underline{\sf SM}(\mathbb{R}^{0|2},X)),
$$
where we've assumed the odd plane bundle on $\Sigma^{d|0}$ is the topologically trivial one. Thus for a fixed connection on $X$, the $S$-points of $\SM(\Sigma^{d|2},X)$ can be identified with quadrupoles
\beq
f\colon S\times \Sigma^{d}\to X, \ \phi_1\in \Gamma( f^*\pi TX), \ \phi_2\in \Gamma( f^*\pi TX), \ F\in \Gamma( f^*TX). \label{highdSpt}
\eeq
When $d=2$, these are the usual component fields for the 2-dimensional sigma model with two supersymmetries. Furthermore, we can restate the data of (\ref{highdSpt}) as
$$
f\in \SM(\Sigma,X), \  \phi_1\in \pi T_f(\SM(\Sigma,X)), \ \phi_2\in \pi T_f(\SM(\Sigma,X)), \ F\in T_f(\SM(\Sigma,X)).
$$
This shows that Lemma \ref{geo} holds for generalized manifolds that arise as mapping spaces.\end{rmk}

\subsection{Integration on $\SM(\R^{0|\delta},X)$}\label{sec:integration}

Integration of functions on $\SM(\R^{0|1},X)$ is particularly easy, owing to a canonically trivialized Berezinian line. Explicitly, integration is the composition
$$
C^\infty(\SM(\R^{0|1},X))\cong \Omega^\bullet(X)\stackrel{\rm project}{\longrightarrow} \Omega^{\rm top}(X)\stackrel{\int}{\to} \R,
$$
where the last arrow requires an orientation on $X$. We claim that a similar situation holds for $\SM(\R^{0|\delta},X)$, $\delta>1$. The key result is the following.

\begin{prop} Let $\delta>0$. Given a choice of connection on $TX$, there is an isomorphism
$$
\SM(\R^{0|\delta},X)\cong \pi T(T^{\delta-1}X)
$$
as supermanifolds. For $\delta>2$ this isomorphism requires a framing of $\R^{0|\delta}$. \label{intprop}\end{prop}

With this proposition, we can define integration (for compactly supported or Schwartz functions, denoted ${\rm cs}$) as
$$
C^\infty_{\rm cs}(\SM(\R^{0|\delta},X))\cong \Omega^\bullet_{\rm cs}(T^{\delta-1}X)\stackrel{\rm project}{\longrightarrow} \Omega^{\rm top}_{\rm cs}(T^{\delta-1}X)\stackrel{\int}{\to}\R.
$$
Furthermore, since $TM$ is canonically oriented for any manifold $M$, this integration map has no topological obstruction on $X$ when $\delta>1$. 

\begin{proof}[Proof of Proposition \ref{intprop}.] We've already proved this for $\delta=1$ (without assuming the existence of a connection). Next we prove the proposition for $\delta=2$. A connection on $X$ splits $T(TX)$ into horizontal and vertical subspaces, 
$$
T(TX)\cong H(TX)\oplus V(TX)\cong p^*(TX\oplus TX)
$$
where we get the second isomorphism from maps $Tp\colon H(TX)\to TX$ and the canonical map $V(TX)\to TX$. Sprinkling in the parity reversal functor we get
$$
\pi T(TX)\cong p^*(\pi (TX\oplus TX))\cong \SM(\R^{0|2},X)
$$
where the second isomorphism uses Lemma \ref{geo} concluding the proof for $\delta=2$. Now we iterate the above isomorphism,
$$
\SM(\R^{0|\delta},X)\cong (\pi T)^\delta X\cong \pi T(T^{\delta-1} X),
$$
where the first isomorphism requires a framing on $\R^{0|\delta}$, and we use the fact that $\pi T\pi TX\cong\SM(\R^{0|2},X)\cong \pi T (TX).$ 
 \end{proof}

\begin{rmk} Following \cite{gorms}, we can show that the following is a trivializing section 
$$
d_1x^1d_1\xi^1\cdots d_1 x^nd_1 \xi^n d_2x^1d_2\xi^1\cdots d_2x^nd_2\xi^n\in \Gamma({\rm Ber}(T\SM(\R^{0|2},U))
$$
which is independent of the choice of coordinates, verifying that the Berezinian of $\SM(\R^{0|2},U)$ is canonically trivialized independent of the choice of connection used above. However, since the integration map from the sigma model uses the metric and connection to define a Gaussian measure, we prefer the more geometric argument above. One can check (e.g., in coordinates) that the two trivializing sections of the Berezinian are in fact equal. 
\end{rmk}

\subsection{The Lagrangian density}\label{sec:Lagdens}

As is usual in Lagrangian mechanics the action functional is defined in terms of a Lagrangian density, which we will state in terms of $S$-points $\Phi$,
$$
\mathcal{S}_h(\Phi)=\int_{S\times \R^{0|2}/S}\mathcal{L}_h(\Phi),\quad \Phi\in \SM(S\times\R^{0|2},X),
$$
where $\mathcal{L}_h\in {\rm Ber}(S\times \R^{0|2}/S)$ is a relative density; see Appendix \ref{berint}. The Lagrangian will contain two terms,
$$
\mathcal{L}_h=\frac{1}{2} \|T\Phi\|^2-\Phi^*h
$$
to be defined below, for $h\in C^\infty X$. In this section we focus on defining and understanding the first term (i.e., set $h=0$). We will first give a coordinate-independent definition of $\mathcal{L}_0$. Then we fix a choice of coordinates, and set up for computations that follow. 

Let $\Phi\in \SM(\R^{0|2},X)(S)$. Then $T\Phi\in \Gamma(S\times \R^{0|2},{\sf Hom}_{S\times\R^{0|2}}(T\R^{0|2},\Phi^*TX)),$ where (in an abuse of notation) $T\R^{0|2}$ denotes the vertical tangent bundle to $S\times \R^{0|2}\to S$. The metric on $X$ gives a pairing
$$
\langle -\rangle \colon  \Phi^*TX\otimes \Phi^*TX\to C^\infty(S\times \R^{0|2}),
$$
which we apply to $T\Phi\otimes T\Phi\in \Gamma(S\times \R^{0|2},{\sf Hom}_{S\times\R^{0|2}}(T\R^{0|2},\Phi^*TX)^{\otimes 2})$ to obtain
$$
2\mathcal{L}_0=\|T\Phi\|^2: =\langle T\Phi\otimes T\Phi\rangle \in \Gamma(S\times\R^{0|2},{\sf Hom}_{S\times\R^{0|2}}(T\R^{0|2}\otimes T\R^{0|2},\underline{\R}))
$$
where $\underline{\R}$ denote the trivial bundle on $S\times\R^{0|2}$. By the symmetry of the pairing $\langle-\rangle$, we find that 
$$
\|T\Phi\|^2\in \Gamma(\Sym^2((T\R^{0|2})^*)\subset \Gamma((T\R^{0|2}\otimes T\R^{0|2})^*)\cong \Gamma({\sf Hom}_{S\times\R^{0|2}}(T\R^{0|2}\otimes T\R^{0|2},\underline{\R}))
$$ 
where $\Sym^2((T\R^{0|2})^*)$ is the second symmetric power of the super vector bundle $(T\R^{0|2})^*$. This bundle is precisely ${\rm Ber}(S\times \R^{0|2}/S)$, verifying that $\|T\Phi\|^2$ is indeed a section of the relative Berezinian. 

If we follow the action of $\Euc(\R^{0|2})$ through the definition of $\|T\Phi\|^2$, we find that it acts on the map entirely through its action on $\Sym^2((T\R^{0|2})^*)$, which in turn is induced from from the action of $\Euc(\R^{0|2})$ on $\R^{0|2}$. To be explicit, the action by translations is trivial, and given an $S$-point $A$ of $O(2)$, it acts by $1/{\rm Det}(A) $ on the bundle $\Sym^2((T\R^{0|2})^*)$. This is the identical action to that of $\Euc(\R^{0|2})$ on ${\rm Ber}(S\times\R^{0|2}/S)$, so the map 
$$
\|T\Phi\|^2\colon  \SM(\R^{0|2},X)\to {\rm Ber}(S\times\R^{0|2}/S)
$$
is $\Euc(\R^{0|2})$-equivariant, and so defines an \emph{invariant} section of ${\rm Ber}(S\times\R^{0|2}/S)$. 

Above we understood $\|T\Phi\|^2$ in terms of the supergeometry of $\SM(\R^{0|2},X)$; presently we wish to describe $\|T\Phi\|^2$ via local geometry on $X$, e.g., curvature and the Riemannian metric. First we trivialize the Berezinian by chosing coordinates $\theta_1,\theta_2$, which gives us trivializing sections $\partial_{\theta_1},\partial_{\theta_2}$ of $T\R^{0|2}$, so
\beq
\Phi &\stackrel{\|T\Phi\|^2}{\mapsto} & \langle T\Phi(\partial_{\theta_1}),T \Phi(\partial_{\theta_2} )\rangle [d\theta_1d\theta_2]\nonumber.
\eeq
Below we will focus on computing $\langle T\Phi(\partial_{\theta_1}),T \Phi(\partial_{\theta_2} )\rangle$. We require the following.
\begin{lem} \label{lem:pullback} Let $f\colon  N\to M$ be a map of supermanifolds. There is an isomorphism of $C^\infty(N)$-modules,
$$
 {\sf Der}(C^\infty M,C^\infty M)\otimes_f C^\infty N\cong {\sf Der}_f(C^\infty M,C^\infty N),$$
 where ${\sf Der}$ denote derivations from an algebra to itself, and ${\sf Der}_f$ denotes derivations between algebras with respect to a morphism~$f$. 
\end{lem} 
\begin{proof} For $W\otimes n\in {\sf Der}(C^\infty M,C^\infty M)\otimes_f C^\infty N$ we define a map
$$W\otimes n \mapsto n\cdot V, \quad V(m):= (f^*W)m.$$ 
One can show that map is bijective abstractly, but we will need an explicit inverse map for computations below. As usual with maps \emph{into} a tensor product, this inverse is somewhat less natural and we need coordinates $\{x^i\}$ on $M$ to define it. We will show the above isomorphism holds in each coordinate patch, and the sheaf property will prove the result. 

With a choice of coordinates in effect, given $V\in {\sf Der}_f(C^\infty M,C^\infty N)$, we get a map
$$
V\mapsto \sum (\partial_{x^i})\otimes_f V(x^i).
$$
One can check explicitly that this defines an inverse in the given chart $\{x^i\}$. \end{proof}

To calculate $\langle T\Phi(\partial_{\theta_1}),T \Phi(\partial_{\theta_2} )\rangle$, we apply Lemma \ref{lem:pullback} to $N=S\times\R^{0|2}$ and $M=X$. For an $S$-point $\Phi\in \SM(\R^{0|2},X)(S)$, we will examine the composition
\beq
\begin{array}{ccc}
\Big( {\sf Der}(C^\infty X,C^\infty(S\times\R^{0|2}))\Big) \bigotimes_{C^\infty(S\times \R^{0|2})} \Big( {\sf Der}(C^\infty X,C^\infty(S\times\R^{0|2}))\Big) \\
\downarrow \cong \\
\Big( {\sf Der}(C^\infty X)\otimes_\Phi C^\infty(S\times \R^{0|2})\Big) \bigotimes_{C^\infty (S\times\R^{0|2})}\Big({\sf Der}(C^\infty X)\otimes_\Phi C^\infty(S\times \R^{0|2})\Big)\\
\downarrow \cong \\
{\sf Der}(C^\infty X)\otimes_{C^\infty X} {\sf Der}(C^\infty X)\otimes_\Phi C^\infty(S\times \R^{0|2})\\
\downarrow g \\
C^\infty X\otimes_\Phi C^\infty(S\times \R^{0|2}) \\
\downarrow {\rm act} \\
C^\infty(S\times\R^{0|2})
\end{array}\nonumber
\eeq
where in the last line we use the action of $C^\infty X$ on $C^\infty(S\times \R^{0|2})$ by $\Phi^*$, and in the second to last line the metric on $X$ is thought of as
$$
g\colon {\sf Der}(C^\infty X)\otimes_{C^\infty X} {\sf Der}(C^\infty X)\to C^\infty X.
$$
Now we compute for an $S$-point $\Phi\colon S\times\mathbb{R}^{0|2}\to X$,
$$
T\Phi(\partial_{\theta_1})=\psi_1+\theta_2 F, \quad T\Phi(\partial_{\theta_2})=\psi_2-\theta_1F,
$$
so that $T\Phi(\partial_{\theta_i})\in {\rm Der}_\Phi(C^\infty X,C^\infty(S\times\mathbb{R}^{0|2}))$. Lemma \ref{lem:pullback} gives us an isomorphism
$$
T\Phi(\partial_{\theta_i})\in {\rm Der}(C^\infty X,C^\infty(S\times\mathbb{R}^{0|2}))\cong {\rm Der}(C^\infty X,C^\infty X)\otimes_\Phi C^\infty (S\times\mathbb{R}^{0|2})
$$ 
and using the proof of the lemma we find
\beq
T\Phi(\partial_{\theta_1})\mapsto \sum_i \frac{\partial}{\partial x^i}\otimes_\Phi (d_1x^i+\theta_2d_2d_1x^i)(\Phi),\nonumber \\
T\Phi(\partial_{\theta_2})\mapsto \sum_j \frac{\partial}{\partial x^j}\otimes_\Phi (d_2x^j-\theta_1d_2d_1x^j)(\Phi),\nonumber 
\eeq
where as usual we are identifying functions with their natural transformations. Then we can apply the pairing $g$,
$$
\langle T\Phi(\partial_{\theta_1}),T\Phi(\partial_{\theta_2})\rangle=\sum_{ij} g_{ij}\otimes_\Phi (d_1x^i+\theta_2d_2d_1x^i)(d_2x^j-\theta_1d_2d_1x^j)(\Phi)
$$
where $g_{ij}$ are the components of $g$ in the given coordinates. 

It remains to understand how the pulled back metric, $\Phi^*g_{ij}$, acts on functions on~$\underline{\sf SM}(\mathbb{R}^{0|2}, M)$, so we compute 
$$
\Phi^*(g_{ij})=g_{ij}(\Phi)+\theta_1d_1g_{ij}(\Phi)+\theta_2d_2g_{ij}(\Phi)+\theta_1\theta_2d_2d_1g_{ij}(\Phi).
$$
Putting this together we obtain an element of $C^\infty(\underline{\sf SM}(\mathbb{R}^{0|2},M))$, whose value at an $S$-point~$\Phi$ is
\beq
\langle T\Phi(\partial_{\theta_1}),T\Phi(\partial_{\theta_2}) \rangle&=&\sum_{i,j} \Bigg( (g_{ij}+\theta_1d_1g_{ij}+\theta_2d_2g_{ij}+\theta_1\theta_2d_2d_1g_{ij})(\Phi)\nonumber \\
&&\phantom{\sum}\cdot (d_1x^i+\theta_2d_2d_1x^i)(d_2x^j-\theta_1d_2d_1x^j)(\Phi)\Bigg).\label{crap}
\eeq
The above formula simples considerably after integrating over the fibers $S\times \R^{0|2}\to S$, which we will do in the next subsection. 

\subsection{The action functional} \label{sec:action}
In this section we prove Lemma \ref{action}. We accomplish this by computing the integral
$$
\mathcal{S}_0(\Phi):=\int_{S\times \R^{0|2}/S}\frac{1}{2}\|T\Phi\|^2=\int_{S\times\R^{0|2}/S}\frac{1}{2} \langle T\Phi(\partial_{\theta_1}),T\Phi(\partial_{\theta_2})\rangle[d\theta_1d\theta_2]
$$
using Equation \ref{crap}. Recall that
$$
\int_{S\times\R^{0|2}/S}\theta_1\theta_2[d\theta_1d\theta_2]=1\in C^\infty (S),
$$
so if we expand Equation \ref{crap} and project to the $\theta_1\theta_2$ component, we get
\beq
\mathcal{S}_0(\Phi)&=&\frac{1}{2}\sum_{i,j}(g_{ij}d_2d_1x^id_2d_1x^j+d_1g_{ij}d_2d_1x^id_2x^j+d_2g_{ij}d_1x^id_2d_1x^j+d_2d_1g_{ij}d_1x^id_2x^j)(\Phi).\nonumber\\
&=&\frac{1}{2}\sum_{i,j} \Big(g_{ij}d_2d_1x^id_2d_1x^j+\frac{\partial g_{ij}}{\partial x^k}d_1x^kd_2d_1x^id_2x^j+\frac{\partial g_{ij}}{\partial x^k}d_2x^kd_1x^id_2d_1x^j\nonumber \\
&&+\frac{\partial g_{ij}}{\partial x^k}d_2d_1x^kd_1x^i d_2x^j+\frac{\partial^2 g_{ij}}{\partial x^k\partial x^l}d_2x^ld_1x^kd_1x^i d_2x^j\Big)(\Phi)\label{eq:actionexp}
\eeq
Using Lemma \ref{geo}, we can interpret the above in terms of more familiar Riemmanian geometry of $X$. For the rest of the argument, formulas will employ the index summation convention. First we collect the terms in (\ref{eq:actionexp}) that have a first derivative of $g_{ij}$ and we observe
\beq
\left(\frac{\partial g_{ij}}{\partial x^k}d_1x^kd_2d_1x^id_2x^j+\frac{\partial g_{ij}}{\partial x^k}d_2x^kd_1x^id_2d_1x^j +\frac{\partial g_{ij}}{\partial x^k}d_2d_1x^kd_1x^i d_2x^j\right)(\Phi)\nonumber\\
=\left(\frac{\partial g_{ij}}{\partial x^k}-\frac{\partial g_{ki}}{\partial x^j}+\frac{\partial g_{kj}}{\partial x^i}\right)d_2d_1x^kd_1x^id_2x^j(\Phi)
=2\Gamma_{ijk}d_2d_1x^kd_1x^id_2x^j(\Phi)\nonumber
\eeq
where $\Gamma_{ijk}$ denotes the Christoffel symbol. Next we notice that for $x^k$ a coordinate, 
\beq
\Hess(\phi_1,\phi_2)(x^k)(\Phi)=-\Gamma_{ij}^k\phi_1(x^i)\phi_2(x^j)=-\Gamma_{ij}^kdx^idx^j(\Phi)\label{eq:hesslocal}
\eeq
using the fact that the second derivative of a coordinate function vanishes. Making the identifications
$$
\phi_1(x^i)=dx^i(\Phi),\quad \phi_2(x^j)=dx^j(\Phi),\quad (F+\Hess(\phi_1,\phi_2))(x^k)=d_2d_1x^k(\Phi),
$$
we compute 
\beq
2\mathcal{S}_0(\Phi)&=&g_{ij}(F+\Hess(\phi_1,\phi_2))(x^i)(F+\Hess(\phi_1,\phi_2))(x^j)\nonumber \\
&&+2\Gamma_{ijk}(F+\Hess(\phi_1,\phi_2))(x^k)\phi_1(x^i)\phi_2(x^j)+\frac{\partial^2 g_{ij}}{\partial x^k\partial x^l}\phi_1(x^i)\phi_2(x^j)\phi_1(x^k)\phi_2(x^l)\nonumber\\
&=&g_{ij}F(x^i)F(x^j)+g_{ij}F(x^i)\Hess(\phi_1,\phi_2)(x^j)+g_{ij}\Hess(\phi_1,\phi_2)(x^i)F(x^j)\nonumber\\
&&+g_{ij}\Hess(\phi_1,\phi_2)(x^i)\Hess(\phi_1,\phi_2)(x^j)+2\Gamma_{ijk}F(x^i)\phi_1(x^j)\phi_2(x^k)\nonumber\\
&&+2\Gamma_{ijk}\Hess(\phi_1,\phi_2)(x^i)\phi_1(x^j)\phi_2(x^k)+\frac{\partial^2 g_{ij}}{\partial x^k\partial x^l}\phi_1(x^i)\phi_2(x^j)\phi_1(x^k)\phi_2(x^l)\nonumber\\
&=&g_{ij}F(x^i)F(x^j)+2\Gamma_{ijk}\Hess(\phi_1,\phi_2)(x^i)\phi_1(x^j)\phi_2(x^k)\nonumber \\
&&+g_{ij}\Hess(\phi_1,\phi_2)(x^i)\Hess(\phi_1,\phi_2)(x^j)+\frac{\partial^2 g_{ij}}{\partial x^k\partial x^l}\phi_1(x^i)\phi_2(x^j)\phi_1(x^k)\phi_2(x^l)\nonumber\\
&=&\langle F,F\rangle+R(\phi_1,\phi_2,\phi_1,\phi_2).\nonumber
\eeq
where in the last two equalities we use Equation \ref{eq:hesslocal} and the standard expression for the curvature tensor in local coordinates. Thus we have
$$
\mathcal{S}_0(\Phi)=\frac{1}{2}\left(\langle F,F\rangle+R(\phi_1,\phi_2,\phi_1,\phi_2)\right),
$$
proving Lemma \ref{action} when $h=0$. 
For $h\ne 0$, we compute
$$
\Phi^* h=f(h)+\theta_1\phi_1(h)+\theta_2\phi_2(h)+\theta_1\theta_2 E(h),
$$
and when we integrate
$$
\int_{S\times\R^{0|2}/S} \Phi^*h[d\theta_1d\theta_2]=E(h)=F(h)+\Hess(\phi_1,\phi_2)h=\langle F,\nabla h\rangle+\Hess(\phi_1,\phi_2)h.
$$
This computation together with the above shows
\beq
\mathcal{S}_{\lambda h}(\Phi)&=&\int_{S\times\R^{0|2}/S} \left(\frac{1}{2}\|T\Phi\|^2-\lambda (\Phi^*h) \right)\nonumber\\
&=&\frac{1}{2}\langle F,F\rangle+\frac{1}{2}R(\phi_1,\phi_2,\phi_1,\phi_2)-\lambda \langle F,\nabla h\rangle-\lambda \Hess(\phi_1,\phi_2)h,\nonumber
\eeq
which concludes the proof of Lemma \ref{action}. 

\begin{rmk} This formula is in agreement with the definition utilized in \cite{strings1,5lectures,freedclass} to obtain the $1|2$-supersymmetric quantum mechanics action functional from the Lagrangian density. For reference, the kinetic term in this action has the form
$$
\mathcal{S}_0(\gamma,\phi_1,\phi_2,F)=\int_\gamma \left(\frac{1}{2}\|\dot{\gamma}\|^2+\frac{1}{2}\langle \nabla_{\dot{\gamma}} \phi_1,\phi_1\rangle+\frac{1}{2}\langle \nabla_{\dot{\gamma}} \phi_2,\phi_2\rangle+\frac{1}{2}R(\phi_1,\phi_2,\phi_1,\phi_2)+\frac{1}{2}\|F\|^2\right)d\gamma
$$
where $\gamma$ is a path in $X$, $\phi_i\in \Gamma(\gamma^*\pi TX)$ and $F\in \Gamma(\gamma^*TX)$. The dimensional reduction from $1|2$ to $0|2$ has the effect of only considering the constant paths in $X$ for which $\dot{\gamma}=0$. This recovers the action functional we derived above.
 \end{rmk}

\subsection{The quantization}\label{sec:geopush}

In this section, we show how the action functional described above determines a Gaussian measure and hence a quantization for renormalizable $0|2$-EFTs, proving Theorem \ref{02push}. We will define a map 
$$
C^\infty_{\rm pol}(\SM(\R^{0|2},X\times Y))\stackrel{p_!}{\longrightarrow} C^\infty_{\rm pol}(\SM(\R^{0|2},Y)),
$$
and show that it restricts to field theories. Consider
\beq
C^\infty_{\rm pol}(\SM(\R^{0|2},X\times Y))&\stackrel{\cdot \exp(-\mathcal{S})\otimes id}{\longrightarrow}& C^\infty(\SM(\R^{0|2},X))\otimes C^\infty_{\rm pol}(\SM(\R^{0|2},Y))\nonumber \\
&\cong& \Omega^\bullet(TX)\otimes C^\infty_{\rm pol}(\SM(\R^{0|2},Y))\nonumber \\
&\stackrel{\rm project\times id}{\longrightarrow}& \Omega^{\rm top}(TX)\otimes C^\infty_{\rm pol}(\SM(\R^{0|2},Y))\nonumber \\
&\stackrel{\frac{1}{N}\int(-) \otimes id}{\dashrightarrow}& C^\infty_{\rm pol}(\SM(\R^{0|2},Y)),\nonumber
\eeq
where the second line uses Proposition \ref{intprop}, and $\mathcal{S}$ is the action of the $0|2$ sigma model on $X$. What remains is to check convergence of this integral. However, by how we've set things up, the image in $\Omega^\bullet(TX)$ consists of functions with polynomial growth in the noncompact direction, as discussed in \ref{sec:funcs}. We claim that the Gaussian measure $\exp(-\mathcal{S})$ allows us to integrate all functions in the image.

It suffices to work locally on $X$ to verify this claim, and furthermore we can set $Y=\pt$ for this part. So let $U\subset (\R^n,g)$ be an (bounded) open submanifold with coordinates $\{x^i\}$. Then $\SM(\R^{0|2},U)\cong U\times\R^n\times R^{0|2n}$. Polynomial functions at an $S$-point $\Phi=(x,\phi_1,\phi_2,F)$ have the form 
$$
G(\Phi)=g(x)P(F)\omega(\phi_1)\eta(\phi_2)\in C^\infty_{\rm pol}(\SM(\R^{0|2},U))
$$
for $P$, $\omega$ and $\eta$ polynomials. First we multiply by $\exp(-\mathcal{S})$, so we have
\beq
G(\Phi)\exp(-\mathcal{S}(\Phi))=e^{-F^2}e^{-R(\phi_1,\phi_2,\phi_1,\phi_2)}g(x)P(F)\omega(\phi_1)\eta(\phi_2) \label{eq:int}
\eeq
Expanding in coordinates as in the previous subsection, we project to the coefficient of $\Pi_{i=1}^nd_1x^id_2x^i$, which we identify with the a section of the line bundle $\Omega^{\rm top}(TU)$. The only problem we might encounter in convergence of the integral is in the $F$-variable. But $P(F)e^{-|F|^2}$ is integrable on $\R^n$ for any metric $g$, which completes the local argument. 

Now we need to show that this map respects the action by $\Euc(\R^{0|2})\cong \R^{0|2}\rtimes O(2)$, and for this we can no longer set $Y=\pt$. 

\begin{prop} The map $p_!$ restricts to a map on $\R^{0|2}\rtimes SO(2)$-invariant functions. \end{prop}

\begin{proof} Since $\exp(-\mathcal{S})$ is $\R^{0|2}\rtimes O(2)$-invariant, we may assume that we are given an element in $\omega\in C^\infty(\SM(\R^{0|2},X\times Y)$ that is invariant and integrable. It suffices to check the claim locally, so together with Fubini's theorem we can restrict to the case that $X=Y=\R$. Choosing a coordinate $x$ on $X=\R$, we Taylor expand $\omega$ to obtain 
\beq
\omega&=&\omega_0+\omega_1d_1x+\omega_2d_2x+\omega_{12}d_1xd_2x,\nonumber\\
\omega_i&=&\eta_i\cdot P_i(d_2d_1x)f_i(x)\nonumber
\eeq
where $\eta_i\in C^\infty(\SM(\R^{0|2},Y))$, $P_i$ is a polynomial, and $f_i\in C^\infty(X)=C^\infty(\R)$. The map~$p_!$ projects to the last term in the Taylor expansion of $\omega$ and integrates over the $x$ and $d_2d_1x$ variables. So proving that $p_! \omega$ is invariant amounts to showing that either $\eta_{12}$ is invariant or that the integral is zero. 

Since $d_1xd_2x$ and $\omega$ are $SO(2)$-invariant, $\omega_{12}$ must be as well. Any function of $x$ and $d_2d_1x$ is also $SO(2)$-invariant, so $\eta_{12}$ must also be $SO(2)$-invariant. 

Now suppose that $\eta_{12}$ is not invariant under one of the $d_i$, say $d_1$. Then for $\omega$ to be invariant (i.e., $d_1\omega=0$) we require that 
$$
d_1(\omega_2)d_2x=d_1(\omega_{12}d_1xd_2x).
$$
Computing we find that this implies $f_{12}(x)=f_2'(x)$ is a total derivative, and so the integral vanishes. Using the same argument for $d_2$, we conclude that $p_!\omega$ is $\R^{0|2}$-invariant, and so $\R^{0|2}\rtimes SO(2)$-invariant
\end{proof}

\begin{proof}[Proof of Theorem \ref{02push}] We have the isomorphism
$$
0|2\EFT_{\rm pol}^\bullet(X\times Y)\cong C^\infty_{\rm pol}(\SM(\R^{0|2},X\times Y))^{\R^{0|2}\rtimes SO(2)},
$$
where the residual $\Z/2$-action remaining from the original $O(2)$ action on each side above recovers the grading via the $\pm 1$-eigenspaces. We have shown there is a map $0|2\EFT_{\rm pol}^\bullet(X\times Y)\to 0|2\EFT_{\rm pol}^\bullet(Y)$, and it remains to analyze the grading shift. Since $\exp(-\mathcal{S})$ is $O(2)$-invariant, we consider the situation for integrable functions invariant under $\R^{0|2}\rtimes SO(2)$ and in the $\pm 1$ eigenspace of the $\Z/2$-action. Again, this is a local computation and by Fubini's Theorem we can restrict to the case that $X=Y=\R$. 

Using the notation from the proof of the previous proposition, we have $p_! \omega= c\cdot \eta_{12}$ for $c\in\R$, so the change in grading is precisely the grade of $f_{12}(x)P_{12}(d_2d_1x)d_2xd_1x$. Since $f_{12}(x)$ is always even and $d_1xd_2x$ is odd, we may assume that $P$ is homogeneous even or odd. If $P(d_2d_1x)$ is odd, the coefficient $c$ will be zero (e.g., by Wick's Lemma since the super degree coincides with the polynomial degree mod 2 and Equation \ref{eq:int}). If $P$ is even, then the degree is lowered by $1={\rm dim}(\R)$, as required.
\end{proof}

Lastly, we fix the normalization constant $N=(2\pi)^{n/2}$, which completes the construction of the quantization. 

\section{The Chern-Gauss-Bonnet Theorem}\label{sec:CGB}

The main ideas in our proof of the Chern-Gauss-Bonnet theorem were already presented in Section 1. Here we address the technical issues in the integration, 
$$
Z_X(g,\lambda h) =\int_{\underline{\sf SM}(\R^{0|2},X)} \exp(-\mathcal{S}_{\lambda h}(\Phi))\frac{\mathcal{D}\Phi}{N}.
$$

\subsection{Evaluating the integrals}\label{sec:int}

First we compute the integral with $\lambda=0$, and then compute it in the limit $\lambda\to\infty$, proving Theorems \ref{thm1} and \ref{thm2} respectively.

\begin{proof}[Proof of Theorem \ref{thm1}] By construction of the trivializing section of the Berezinian line in  Section \ref{sec:geopush}, integration first projects onto a particular subspace of functions on $\SM(\R^{0|2},X)$. We claim this projects to the correct component of $e^{-R(\phi_1,\phi_2,\phi_1,\phi_2)/2}$ to obtain the Pfaffian. Explicitly, we compute in coordinates from Section \ref{sec:action}
$$
d_2d_1g_{ij}dx^idx^j=R_{klij}d_1x^kd_2x^ld_1x^id_2x^j
$$
so if $\int (-) \mathcal{D}\Phi$ denotes the projection to the component of $d_1x^1d_2x^2\cdots d_1x^nd_2x^n$, 
$$
\int e^{-R(\phi_1,\phi_2,\phi_1,\phi_2)/2} \mathcal{D}\Phi= \frac{(-1)^{n/2}}{(n/2)!2^{n/2}} \sum \epsilon_{i_1 \cdots i_n} R_{i_1i_2i_1i_2}\cdots R_{i_{n-1}i_ni_{n-1}i_n}=\Pf(R).
$$ 
Next, using Lemma \ref{geo}, we identify $\Phi=(x,\phi_1,\phi_2,F)$ and first integrate out the $F$ variable. This is an ordinary Gaussian integral
$$
Z_X(g,0)=\int_{\SM(\R^{0|2},X)} \exp(-\mathcal{S}_0)\frac{\mathcal{D}\Phi}{N}=\frac{1}{N}\int_X \Pf(R) \int_{TX/X} e^{-F^2/2}dF=\int_X\Pf(R).
$$
So, recalling that $N=(2\pi)^{n/2}$, we get one side of the Chern-Gauss-Bonnet formula. \end{proof}

\begin{proof}[Proof of Theorem \ref{thm2}] It remains to compute the integral in the $\lambda\to \infty$ limit. First we integrate out the $F$-variable 
\beq
\int_{\SM(\R^{0|2},X)} \exp(-\mathcal{S}_{\lambda h})\frac{\mathcal{D}\Phi}{N}&=&\frac{1}{N}\int_{\pi TX\oplus \pi TX}\Bigg( \exp\left(-R(\phi_1,\phi_2,\phi_1,\phi_2)/2+\lambda\Hess(h)(\phi_1,\phi_2)\right)\nonumber \\
&&\cdot \int_{\stackrel{\SM(\R^{0|2},X)/}{\pi TX\oplus \pi TX}} \exp\left(-\langle F,F\rangle/2-\lambda \langle \nabla h,F\rangle\right)\Bigg)\nonumber
\eeq
so we compute the Gaussian integral
$$
\int_{\R^n} e^{-\langle F,F\rangle/2-\lambda \langle \nabla h,F\rangle}dF=(2\pi)^{n/2}e^{-\frac{\lambda^2}{2}\|\nabla h\|^2}.
$$
Now we employ an argument similar to that of Mathai-Quillen \cite{mathai-quillen}. First, we assume that $h$ is Morse, and choose small disjoint open neighborhoods $U_p$ of the critical points of~$h$. Let $X^c:=X-\bigcup U_p$. Then
$$
\int_{\SM(\R^{0|2},X)} \exp(-\mathcal{S}_{\lambda h})\frac{\mathcal{D}\Phi}{N}=\int_{\SM(\R^{0|2},X^c)} \exp(-\mathcal{S}_{\lambda h})\frac{\mathcal{D}\Phi}{N}+\int_{\SM(\R^{0|2},\bigcup U_p)} \exp(-\mathcal{S}_{\lambda h})\frac{\mathcal{D}\Phi}{N}
$$
We know that $\|\nabla h\|^2>0$ has a lower bound on $X^c$, so as $\lambda\to \infty$ we find
\beq
\int_{\SM(\R^{0|2},X^c)} \exp(-\mathcal{S}_{\lambda h})\mathcal{D}\Phi&=&\int_{\pi TX^c\oplus \pi TX^c} \exp\Bigg(-\frac{\lambda^2}{2}||\nabla h||^2\nonumber \\
&& \phantom{hello} +\lambda\Hess(h)(\phi_1,\phi_2)-R(\phi_1,\phi_2,\phi_1,\phi_2)\Bigg) \to 0.\nonumber
\eeq
It remains to evaluate the integral near the critical points of $h$. Focusing our attention on one such point $p$ (and possibly shrinking $U_p$), we choose coordinates on $U_p$ and via a concordance deform the metric to the standard one on $\R^n$. By Theorem \ref{thm3}, this concordance does not affect the value of the integral. The standard metric is flat, so we obtain
$$
\int_{\SM(\R^{0|2},U_p)} \exp(-\mathcal{S})\mathcal{D}\Phi=\int_{\SM(\R^{0|2},U_p)}\exp\left(-\frac{\lambda^2}{2}\|\nabla h\|^2+\lambda\Hess(h)(\phi_1,\phi_2)\right)
$$
The remainder of the computation amounts to a pair of Gaussian integrals. The Berezinian integration is a fermionic Gaussian integral (see Appendix \ref{berint}) with respect to the pairing~$\Hess(h)$, and we find
$$
\int_{\pi TU_p\oplus \pi TU_p} \exp(\lambda \Hess(h)(\phi_1,\phi_2))\mathcal{D}\Phi={\lambda^n}{\rm Det}(\Hess(h)),$$
where the right hand side is understood to be a differential form on $U_p$. Hence,
$$
\int_{\SM(\R^{0|2},U_p)} \exp(-\mathcal{S})\mathcal{D}\Phi=\lambda^n \int_{U_p} \exp\left(-\frac{\lambda^2}{2}||\nabla h||^2\right){\rm Det}(\Hess(h)).
$$
There are coordinates $\{x^i\}$ where the vector field $\nabla h$ on $U_p$ can be represented by a matrix~$H_p$, where we get a vector field on $\R^n$ by $x\mapsto H_px$. Note that in these coordinates~$\Hess(h)=H_p$ is  symmetric and nondegenerate. We can repackage the above as 
$$
\int_{\SM(\R^{0|2},U_p)} \exp(-\mathcal{S})\mathcal{D}\Phi= \lambda^n {\rm sgn}({\rm Det} H_p) \int_{U_p} \exp\left(-\frac{\lambda^2}{2}\|H_px\|^2\right)|{\rm Det}(H_p)|dx^1\wedge\dots \wedge dx^n.
$$
This is an (ordinary) Gaussian integral, and in the limit $\lambda\to \infty$ the value of the integral on $U_p$ approaches the value of the integral on $\R^n$, and we obtain
$$
\lim_{\lambda\to \infty}\int_{\SM(\R^{0|2},U_p)} \exp(-\mathcal{S})\frac{\mathcal{D}\Phi}{N}= {\rm sgn}({\rm Det}\Hess(h)).
$$
Summing over critical points
$$
\lim_{\lambda\to\infty}Z_X(g,\lambda h)=\sum_{p\in {\rm zero}(dh)} {\rm sgn}({\rm Det} \Hess(h))= {\rm Index}(\nabla h)=\chi(X),
$$
proving the result. 
\end{proof}

\appendix
\section{Supermanifold miscellany} \label{berint} \label{appensuper}

\subsection{Berezinian integrals by example}
For a throughout treatment of (relative) Berezinian integrals, see the article by Deligne and Morgan \cite{strings1}. Here we give a few examples. 

\begin{ex}[Relative integration] The important case to consider is relative integration for the trivial bundle $\R^{n|m}\times S\to S$. For any other family $M\to S$, the relative integration is locally of this form. On this bundle, the relative Berezinian is an $C^\infty_{\R^{n|m}}$-module of rank $1|0$ if $m$ is even, and rank $0|1$ if $m$ is odd. A choice of coordinates $\theta_1,\dots,\theta_n$ induces a trivialization of this module, and we denote the trivializing section by $[d\theta_1\cdots d\theta_n]$. If we tensor the relative Berezinian with the relative orientation bundle, we get a map
$$
\int_{\R^{n|m}\times S/S}\colon C^\infty(\R^{n|m}\times S)\to C^\infty(S). 
$$
Consider first the case where $n=0$. To Evaluate on a function, we Taylor expand in $\theta_i$ and project to the component of $\theta_1\cdots\theta_n$, obtaining a function on $S$. When $n\ne 0$, first we project, obtaining a function on $C^\infty(\R^n\otimes S)$. Next we use an orientation form on $\R^n$ to integrate and obtain a function on $C^\infty S$. 
\end{ex}

\begin{ex}[Fermionic Gaussians] The following is standard and can be found, for example, in \cite{susyequivderham}. 
Let $q$ be a quadratic form on a purely odd supervector space $V$ of even dimension. Note that for vectors $\omega,\eta\in V$, super quadratic means
$$
q(\omega,\eta)=-q(\eta,\omega). 
$$
Thinking of functions on $V$ as being an exterior algebra, we see $q$ is in the 2nd antisymmetric power, so in particular, $q\in C^\infty(V)$. We claim
$$
\int_V\exp\left(-\frac{1}{2}\tilde{q}\right)={\rm Pf}(q)={\rm Det}(q)^{1/2}.
$$
To see this, choose coordinates $C^\infty(V)\cong \R[\theta_1,\dots,\theta_{2n}]$ such that $q$ is a skew matrix of the form
\beq
q=\left(\begin{array}{cccccccccccccccc} 
\lambda_1 J & 0 & \dots & 0 \\
0 & \lambda_2 J & \dots & 0 \\
\vdots & \vdots & \ddots & \vdots \\
0 & 0 & \dots & \lambda_n J
\end{array}\right),\quad J=\left(\begin{array}{cc} 0 & -1 \\ 1 & 0\end{array}\right).
\eeq
The Berezinian integral projects to the top component of $\exp(-q/2)$, which is 
$$
(-2)^n\frac{1}{n!} q^n=\lambda_1\cdots\lambda_n\theta_1\cdots \theta_{2n}.
$$
Thus, the value of the integral is the product of the $\lambda_i$, which is precisely the Pfaffian of $q$. To compare, the determinant of $q$ is 
$$
{\rm Det}(q)=\lambda_1^2\cdots \lambda_n^2,
$$
so in particular, the integral singles out a preferred square root of ${\rm Det}(q)$. This verifies the claim. 

\begin{rmk}
Compare the above with the usual Gaussian integral,
$$
(2\pi)^{m/2}\int_W \exp\left(-\frac{1}{2}\tilde{q}\right)=\frac{1}{{\rm Det}(q)^{1/2}}.
$$
for $W$ an even (i.e., bosonic) $m$-dimensional vector space.
\end{rmk}

\end{ex}

\subsection{Functions on (generalized) supermanifolds}\label{functions}

One can identify the functor $C^\infty$ that takes a supermanifold $M$ to its superalgebra of functions with the functor $M\mapsto {\sf SM}(M,\R^{1|1})$ where addition and multiplication on the image are defined using addition and (the commutative) multiplication on $\R^{1|1}$. The grading on this algebra comes the the involution $\alpha$ of $\R^{1|1}$ determined by
$$
\alpha^*\colon C^\infty(\R)[\theta]\to C^\infty(\R)[\theta], \quad \theta\mapsto -\theta
$$
where we have identified $C^\infty(\R^{1|1})\cong C^\infty(\R)[\theta]$. Using the Yoneda Lemma, the morphisms~${\sf SM}(M,\R^{1|1})$ are determined by natural transformations between the functors $\underline{M}$ and $\underline{\R}^{1|1}$. Such a natural transformation is a map of sets
$$
{\sf SM}(S,M)\to {\sf SM}(S,\R^{1|1})\cong C^\infty(S).
$$
Hence, maps of sets $\underline{M}(S) \to C^\infty(S)$ natural in $S$ are in bijection with functions on $M$. This makes sense for $M$ a generalized supermanifold. Being a functor valued in algebras, generalized supermanifolds have an algebra of functions. In fact, since objects in the category of generalized super manifolds can be written as a coequalizer of supermanifolds and equalizers of nuclear vector spaces are nuclear, we obtain a nuclear super vector space of functions on a generalized supermanifold. This fact was explained to me by Dmitri Pavlov.

\bibliographystyle{amsalpha}
\bibliography{references}

\newcommand{\etalchar}[1]{$^{#1}$}
\providecommand{\bysame}{\leavevmode\hbox to3em{\hrulefill}\thinspace}
\providecommand{\MR}{\relax\ifhmode\unskip\space\fi MR }
\providecommand{\MRhref}[2]{%
  \href{http://www.ams.org/mathscinet-getitem?mr=#1}{#2}
}
\providecommand{\href}[2]{#2}
\begin{thebibliography}{DEF{\etalchar{+}}99}

\bibitem[AD99]{driver}
L.~Andersson and B.~Driver, \emph{Finite dimensional approximations to {Wiener}
  measure and path integral formulas on manifolds}, Journal of Functional
  Analysis \textbf{165} (1999), 430--498.

\bibitem[AG83]{Alvarez}
L.~Alvarez-Gaum{\'e}, \emph{Supersymmetry and the {Atiyah-Singer} index
  theorem}, Communications in Mathematical Physics \textbf{90} (1983),
  161--173.

\bibitem[BGV92]{BGV}
N.~Berline, E.~Getzler, and M.~Vergne, \emph{Heat kernels and {Dirac}
  operators}, Springer, 1992.

\bibitem[BP08]{baer}
C.~{B\"ar} and F.~Pf{\"a}ffle, \emph{Path integrals on manifolds by finite
  dimensional approximation}, Journal {f\"ur} die reine und angewandte
  Mathematik \textbf{625} (2008), 29--57.

\bibitem[Che45]{CGB}
S.S. Chern, \emph{On the curvatura integra in riemannian manifold}, Annals of
  Mathematics \textbf{46} (1945), 674Ð684.

\bibitem[Cos11]{costbook}
K.~Costello, \emph{Renormalization and effective field theories}, American
  Mathematical Society, 2011.

\bibitem[DEF{\etalchar{+}}99]{strings1}
P.~Deligne, P.~Etingof, D.~Freed, L.~Jeffrey, D.~Kazhdan, J.~Morgan,
  D.~Morrison, and E.~Witten, \emph{Quantum fields and strings: {A} course for
  mathematicians, volume 1}, American Mathematical Society, 1999.

\bibitem[Fre99]{5lectures}
D.~Freed, \emph{Five lectures on supersymmetry}, American Mathematical Society,
  1999.

\bibitem[Fre01]{freedclass}
Dan Freed, \emph{Classical field theory and supersymmetry}, {IAS}/Park City
  Mathematics Series \textbf{11} (2001).

\bibitem[GS99]{susyequivderham}
V.~Guillemin and S.~Sternberg, \emph{Supersymmetry and equivariant de rham
  theory}, Springer, 1999.

\bibitem[HKST11]{HKST}
H.~Hohnhold, M.~Kreck, S.~Stolz, and P.~Teichner, \emph{Differential forms and
  0-dimensional super symmetric field theories}, Quantum Topol \textbf{2}
  (2011), 1Ð41.

\bibitem[HST10]{mingeo}
H.~Hohnhold, S.~Stolz, and P.~Teichner, \emph{From minimal geodesics to super
  symmetric field theories}, CRM Proceedings and Lecture Notes \textbf{50}
  (2010).

\bibitem[HST11]{STsuper}
Henning Hohnhold, Stephan Stolz, and Peter Teichner, \emph{Supermanifolds: an
  incomplete survey}, Bulletin of the Manifold Atlas (2011), 1--6.

\bibitem[Kon03]{kont}
Maxim Kontsevich, \emph{Deformation quantization of {Poisson} manifolds, {I}},
  Lett. Math. Phys. \textbf{66} (2003), 157--216.

\bibitem[K{\v S}04]{gorms}
D.~Kochan and P.~{\v S}evera, \emph{Differential gorms, differential worms},
  preprint (2004).

\bibitem[Lot87]{lott_susy}
J.~Lott, \emph{Supersymmetric path integrals}, Communications in Mathematical
  Physics \textbf{108} (1987), 605--629.

\bibitem[MQ86]{mathai-quillen}
V.~Mathai and D.~Quillen, \emph{Superconnections, {Thom} classes and
  equivariant differential forms}, Topology \textbf{25} (1986).

\bibitem[Roe98]{roe}
J.~Roe, \emph{Elliptic operators, topology and asymptotic methods}, CRC Press,
  1998.

\bibitem[ST11]{ST11}
S.~Stolz and P.~Teichner, \emph{Supersymmetric field theories and generalized
  cohomology}, Mathematical Foundations of Quantum Field and Perturbative
  String Theory ({B. Jur{\v c}o, H. Sati, U. Schreiber}, ed.), Proceedings of
  Symposia in Pure Mathematics, 2011.

\bibitem[Wei09]{weinstein}
A.~Weinstein, \emph{{The Volume of a Differentiable Stack}}, Letters in
  Mathematical Physics \textbf{90} (2009), 353--371.

\bibitem[Wit82]{susymorse}
E.~Witten, \emph{Supersymmetry and {Morse} theory}, Journal of Differential
  Geometry \textbf{17} (1982), 661--692.

\bibitem[Wit88]{witten_dirac}
\bysame, \emph{The index of the {Dirac} operator in loop space}, Lecture Notes
  in Mathematics \textbf{1326} (1988), 161--181.

\end{thebibliography}
\end{document}